\documentclass[11pt,letterpaper]{amsart} 
\usepackage{amssymb,latexsym,amsmath}
\usepackage[mathscr]{eucal}
\usepackage[margin=1in]{geometry}
\usepackage{enumitem}

\AtBeginDocument{%
   \def\MR#1{}
}

\newtheorem{theorem}{Theorem}[section]
\newtheorem{lemma}[theorem]{Lemma}
\newtheorem{proposition}[theorem]{Proposition}

\theoremstyle{definition}
\newtheorem{remark}[theorem]{Remark}

\numberwithin{equation}{section}

\newcommand{\RR}{\ensuremath{\mathbb{R}}}

\newcommand{\prtl}{\ensuremath{\partial}}
\newcommand{\veps}{\ensuremath{\varepsilon}}
\newcommand{\Ocal}{\ensuremath{\mathcal{O}}}
\newcommand{\U}{\ensuremath{\mathcal{U}}}
\newcommand{\V}{\ensuremath{\mathcal{V}}}
\newcommand{\h}{\hbar}
\newcommand{\Lsc}{\mathscr{L}}
\newcommand{\Msc}{\mathscr{M}}

\newcommand{\supp}{{\text{\rm supp}}}

\renewcommand{\Re}{\text{\rm Re}\,}
\renewcommand{\Im}{\text{\rm Im}\,}

\renewcommand{\d}{d}

\title[The Van Vleck Formula on Ehrenfest time scales]{The Van Vleck Formula on Ehrenfest time scales and stationary phase asymptotics for frequency-dependent phases}

\author[M. D. Blair]{Matthew D. Blair}
\address{Department of Mathematics and Statistics, University of New Mexico, Albuquerque, NM, USA}
\email{blair@math.unm.edu}

\begin{document}
\begin{abstract} 
The Van Vleck formula is a semiclassical approximation to the integral kernel of the propagator associated to a time-dependent Schr\"odinger equation.  Under suitable hypotheses, we present a rigorous treatment of this approximation which is valid on \emph{Ehrenfest time scales}, i.e. $\h$-dependent time intervals which most commonly take the form $|t| \leq c|\log\h|$.  Our derivation is based on an approximation to the integral kernel often called the \emph{Herman-Kluk approximation}, which realizes the kernel as an integral superposition of Gaussians parameterized by points in phase space.  As was shown by Robert  \cite{RobertHK}, this yields effective approximations over  Ehrenfest time intervals.  In order to derive the Van Vleck approximation from the Herman-Kluk approximation, we are led to develop stationary phase asymptotics where the phase functions depend on the frequency parameter in a nontrivial way, a result which may be of independent interest.
\end{abstract}
\maketitle

\section{Introduction}\label{S:intro}
Let $\hat{H}(t)$ be a self-adjoint quantum Hamiltonian defined by $\h$-Weyl quantization of a real-valued, $C^\infty$ symbol $H(t,\cdot) $ on $T^*\RR^d \cong \RR^{2d}$ (cf. \cite{RobertSemiClassique}, \cite{ZworskiSemiclassicalAnalysis}), and consider the time-dependent Schr\"odinger equation
\begin{equation}\label{schrodIVP}
	i\h \frac{\prtl\psi }{\prtl t}(t) = \hat{H}(t)\psi(t), \quad \psi(t)|_{t=0} = \psi_0 \in L^2(\RR^d).
\end{equation}
As long as $\hat{H}(t)$ is reasonably well-behaved, there exists a family of unitary operators $\U_t$ on $L^2(\RR^d)$ such that $\psi(t) = \U_t\psi_0$ gives the unique solution to \eqref{schrodIVP}.  If $\hat H$ is independent of $t$, then $\U_t = e^{-\frac {it}\h \hat H}$ as a semigroup. Otherwise, one can also consider unitary operators $\U_{t,s}$ generating solutions to the initial value problem at some other time $s$, so that $\U_{s,s}\psi_0 = \psi_0$, but here we restrict attention to the former without any loss of generality.

In what follows, we use $z$ to denote variables in $\RR^{2d} \cong T^*\RR^d$ (the cotangent bundle of $\RR^d$) and often decompose such elements as $z=(q,p) \in \RR^d \times \RR^d$ so that $q,p$ give position and momentum variables respectively. Let $H_0(t,z):\RR \times\RR^{2d} \to \RR$ denote the principal symbol of $H(t,z)$.  The Hamiltonian vector field of $H_0$ generates a flow defined by the integral curves of
\begin{equation}\label{HamEqn}
	\dot q_t = \frac{\prtl H_0}{\prtl p}(t,q_t,p_t), \quad \dot p_t =-\frac{\prtl H_0}{\prtl q}(t,q_t,p_t), \quad (q_t,p_t)|_{t=0} = (q,p).
\end{equation}
Throughout this work, we assume the corresponding flow on $\RR^{2d}$ is complete in that the maximal domain of every integral curve is all of $\RR$. We denote the flow by $\kappa_t(q,p) = (q_t(q,p),p_t(q,p))$, so that the left refers to the flow on $\RR^{2d}$ and the right gives expresses its position and momentum components.  We let $S$ denote the associated classical action $S: \RR \times \RR^{2d} \to \RR$ given by
\begin{equation}\label{actionintro}
	S(t,q,p) = \int_0^t p_s \cdot \dot{q}_s  - H_0\big(s,q_s(q,p) ,p_s(q,p) \big) \,ds .
\end{equation}

The Van Vleck formula is a semiclassical approximation to the distributional kernel of $\U_t$, denoted here as $K_t(x,y)$, in terms of the classical paths joining $y$ to $x$, i.e. solutions to \eqref{HamEqn} such that $q_t(y,p) =x$. To this end, it is typically assumed that
\begin{equation}\label{nonsinghypintro}
	\frac{\prtl q_t}{\prtl p}(y,p)\Big|_{p=\eta} \text{ is nonsingular whenever } \eta \in \Xi_t (x,y):= \{p \in \RR^d: x=q_t(y,p)\}.
\end{equation}
When $\Xi_t(x,y) \neq \emptyset$, the Van Vleck formula roughly states that to leading order in $\h$,
\begin{equation}\label{VVleadingorder}
	K_t(x,y) \sim 
	\frac{1}{(2\pi \h)^{\frac d2}}\sum_{ \eta \in \Xi_t (x,y)} e^{\frac i\h S_\eta-i\frac{\pi}{2} \theta_\eta} \Big|\det \frac{\prtl^2 S_\eta}{\prtl x \prtl y}\Big|^{1/2} \quad \text{ as }\h\to 0.
\end{equation}
Here $S_\eta$ denotes $S(t,y,\eta)$, the classical action along the path from $(y,\eta)$ and $\theta_\eta \in \frac 12 \mathbb{Z}$ denotes the (half-integer) Maslov index\footnote{Here we treat the Maslov index in a manner consistent with \cite{MeinrenkenGutzwillerTrace}, allowing it to take on a half-integer value.  This means the expression here differs slightly from that in some other treatments, where the first factor is $(2\pi i\h)^{-\frac d2}$ instead of $(2\pi \h)^{-\frac d2}$ as we have here.  This convention also agrees with the one in \cite{RobbinSalamonPaths}, where half-integer indices naturally appear for paths starting from the Maslov cycle.} of this path.  Note that  $\frac{\prtl^2 S_{\eta}}{\prtl x \prtl y}$ is well-defined since \eqref{nonsinghypintro} implies that $x=q_t(y,p)$ locally defines $p$ as a function of $(x,y)$ near any $\eta \in \Xi_t(x,y)$ (cf. \S\ref{SS:SPandVVcon} below).

For times $t$ sufficiently close to zero, the approximation \eqref{VVleadingorder} can often be rigorously justified by the WKB approach: this is essentially done in the original work of Van Vleck \cite{VanVleckCorrespondence}.  For larger times in a finite interval $t \in [-T,T]$, obtaining approximations to $K_t(x,y)$ is considerably more involved.  The methods of microlocal analysis furnish oscillatory integrals which approximate $K_t$, at which point \eqref{VVleadingorder} results from an application of the principle of stationary phase. We are aware of a few effective methods to this end:
\begin{itemize}	
	\item \textbf{The canonical operators of Maslov and the global theory of Fourier integral operators}. See the work of Maslov and Fedoriuk \cite[Theorem 12.5]{MaslovFedoriuk} and Meinrenken \cite[p.293]{MeinrenkenGutzwillerTrace} for these respective approaches, which are closely related.  The latter is rooted in work of H\"ormander and Duistermaat, but both approaches essentially develop approximations to $K_t(x,y)$ as Lagrangian distributions, i.e. distributions defined by sums of oscillatory integrals whose wave front sets are contained in the graph of $\kappa_t$ (and the phases taken here are typically real-valued).

	\item \textbf{Fourier integral operators with complex phase}. The work of Laptev and Sigal \cite[Theorem 5.1]{LaptevSigal} shows that by considering Fourier integral operators with complex phase, one can obtain an effective approximation to $K_t(x,y)$ by considering a single oscillatory integral over phase space.  In contrast, the canonical operator and global FIO methods referenced above relied on real phases, which typically realizes the kernel as a sum of oscillatory integrals, which can result in greater technical expense.
	\item \textbf{Coherent state/wave packet methods}.  The coherent state approach of Bily and Robert \cite[Theorem 3.5]{BilyRobert} begins with the results of Combescure and Robert \cite{CombescureRobertWavePackets} (see also \cite{HagedornJoye}) giving accurate approximations to the evolution of a Gaussian under $\U_t$. The kernel $K_t(x,y)$ is then realized as an integral superposition of such Gaussians.  This essentially results in a Fourier integral operator with complex phase, and as in \cite{LaptevSigal}, the approximation involves a single oscillatory integral.  The work \cite{BilyRobert} is particularly influential in the present one; as noted below, the approach taken here is a variation on theirs.
\end{itemize}

The purpose of work is to rigorously investigate the validity of the Van Vleck approximation on $\h$-dependent time scales, namely those which are comparable to the so-called \emph{Ehrenfest time}.  We leave this as a loosely defined term, but nonetheless motivate it by considering the evolution of a Gaussian wave packet.  If the classical Hamiltonian $H_0$ above is well-behaved, then typically the first derivatives of the corresponding flow $\kappa_t$ at least satisfy an exponential upper bound of the form $|\prtl \kappa_t(z)| \lesssim e^{\Lambda|t|}$.  If the flow $\kappa_t$ is chaotic, then such a bound expects to be saturated in some sense, i.e. Hamiltonian rays may diverge at that rate, or at least close to it.  This in turn limits the time scales over which the evolution of a quantum mechanical wave packet expects to be governed by quantities determined by a single classical trajectory.
For example, one typically thinks of a Gaussian  $(\pi\h)^{-\frac d2} e^{\frac i\h (p_0\cdot (q-q_0)+\frac i2 |q-q_0|^2)}$ as concentrating its mass within a region $|q-q_0| \lesssim \h^{\frac 12}$ in position and $|p-p_0| \lesssim \h^{\frac 12}$ in Fourier space (defined by the $\h$-semiclassical Fourier transform).  Thus if rays are diverging at the rate of $e^{\Lambda |t|}$,  the image of the Gaussian under $\U_t$ expects to be concentrated in a region $|q-q_t| \lesssim e^{\Lambda |t|}\h^{\frac 12}$, $|p-p_t| \lesssim e^{\Lambda |t|}\h^{\frac 12}$ where $(q_t,p_t)$ solves \eqref{HamEqn} with $(q_t,p_t)|_{t=0} = (q_0,p_0)$.  So if $|t| \gg |\log \h|$, the image of the Gaussian under $\U_t$ may no longer be concentrated in regions near $(q_t,p_t)$ which shrink as $\h\to0$.  See the introduction of Schubert, Vallejos, and Toscano \cite{SchubertValTos} for a nice description of this phenomena involving the Wigner function.  Here the ``Ehrenfest time'' is thus considered to be a fraction of $|\log\h|$ and represents the largest time scales over which a single trajectory expects to govern wave propagation.  While we have motivated these time scales using wave packets, similar delocalization phenomenon may affect the validity of all the approaches above. We also note that if the flow satisfies a much more stable bound of $|\prtl \kappa_t(z)| \lesssim (1+|t|)^k$, then the Ehrenfest time may be much longer, say $|t| \ll \h^{-\frac 1{2k} +\veps}$.

In the case where $\hat H$ is independent of $t$ and has discrete spectrum, one can of course consider $L^2$ solutions to the stationary equation $E \psi = \hat H \psi$.  In such cases, it is interesting to study the operators of the form $g(\h^{-1}(E-\hat H))$ for some Schwartz class function $g$.  The evolution operators $\U_t$ give a means of analyzing them through the operator-valued Fourier integral
\begin{equation}\label{fourierres}
	g\Big( \frac{E-\hat H}{\h} \Big) =  \frac{1}{2\pi} \int e^{\frac{itE}{\h}} e^{-\frac{it}{\h}\hat H} \hat{g}(t) \,dt = \frac{1}{2\pi} \int e^{\frac{itE}{\h}} \U_t \,\hat{g}(t) \,dt.
\end{equation}
To make use of this identity, it is typically assumed that $\hat{g}$ is compactly supported.  The Van Vleck formula thus gives a lens in to the trace or even the distributional kernel of such operators.  Indeed, since there are uniform approximations to the kernel $K_t(x,y)$ of $\U_t$ on bounded time intervals, these can be inserted into the integral expression above to analyze the kernel of $g(\h^{-1}(E-\hat H))$.  This type of approach was originally suggested by Gutzwiller \cite{gutzwiller1971periodic} in his treatment of the celebrated semiclassical trace formula which often bears his name.  The aforementioned work \cite{MeinrenkenGutzwillerTrace} rigorously shows the trace formula in this manner.  Loosely speaking, the Gutzwiller trace formula yields asymptotics on the trace of $g(\h^{-1}(E-\hat H))$ in terms of the periodic orbits generated by the Hamiltonian flow of $H(z)$, and has origins in the Selberg trace formula.  Since $g$ is rapidly decreasing, the trace reflects the distribution of the spectrum near $E$ as $\h \to 0$.  In particular, the trace formula is typically stronger than the Weyl formula over an interval $I \subset \RR$
\begin{equation}\label{WeylIntro}
		\#(\text{spec}(\hat{H})\cap I) = (2\pi \h)^{-d} \text{Vol} \{H^{-1}(I)\} + O(\h^{1-d}).
\end{equation}

At this stage, we recall the parallel line of work for problems involving eigenfunctions of the Laplacian on a compact Riemannian manifold.  The relation begins with the elementary observation that multiplying the Helmholtz operator $\lambda^2+ \Delta $ by $\h^2$, where $\h = \lambda^{-1}$, results in the semiclassical operator $1+\h^2 \Delta $, and hence this embeds into the analysis above by setting $\hat H = -\h^2 \Delta$ and $E=1$.  In this way, the trace formula above reflects the spectral asymptotics of the Laplacian in the high frequency limit $\lambda \to \infty$.  The works of Chazarain \cite{ChazarainPoisson} and Duistermaat and Guillemin \cite{DuistermaatGuillemin} give rigorous asymptotics on the trace of $g(\lambda^{-1}(\lambda^2-\Delta))$ as $\lambda \to \infty$.

Since $g$ is assumed to be rapidly decreasing, the operators $g(\h^{-1}(E-\hat H))$ in some sense favors the eigenspaces corresponding to the part of the spectrum near $E$.  Moreover, functions in the range of $g(\h^{-1}(E-\hat H))$ give families of approximate eigenfunctions (quasimodes) in that 
\begin{equation*}
	\|(E-\hat H) g(\h^{-1}(E-\hat H)) \|_{L^2(\RR^\d) \to L^2(\RR^\d)} = O(\h), \quad \text{ as } \h \to 0.
\end{equation*}
However, these are arguably very weak approximations as the $O(\h)$ bound here generally cannot be replaced by any larger power of $\h$.  In light of the observations above, one natural way to improve upon this approximation is to instead consider functions $g(|\log\h|\h^{-1} (E-\hat H))$, which improves the $O(\h)$ error to $O(\h/|\log\h|)$. If $\hat{g}$ is of sufficiently small compact support, generalizing \eqref{fourierres} gives
\begin{equation}\label{fourierreslog}
	g\Big( \frac{|\log\h|}{\h} (E-\hat H)\Big) =  \frac{1}{2\pi|\log\h|} \int e^{\frac{itE}{\h}} \U_t \,\hat{g}(t/|\log\h|) \,dt.
\end{equation}
Exploiting this formula effectively thus places a premium on obtaining uniform approximations to $K_t(x,y)$ on time scales $|t| \ll |\log\h|$.  For the Laplacian on a compact Riemannian manifold, this has been achieved in works of B\'erard \cite{Berard77} and Volovoy \cite{Volovoy}. These works give geometric and dynamical conditions that improve the $O(\h^{1-d})$ error in \eqref{WeylIntro} to $O(\h^{1-d}/|\log\h|)$ by proceeding as in \eqref{fourierreslog} (though they do not give higher order asymptotics).

The Van Vleck formula is therefore significant in the analysis of both time-dependent and stationary Schr\"odinger equations.  Taken together, this motivates our interest in developing uniform approximations to $K_t(x,y)$ on Ehrenfest time scales, which at present is admittedly more mathematical than physical.  In particular, it would be interesting to examine trace formulae for the operators in \eqref{fourierreslog} as well as complete asymptotic expansions of their kernels.  Recent progress on both of these problems is due to Canzani and Galkowski \cite{canzani2020weyl}, though their asymptotics are of lower order and the geodesic beam techniques developed there do not explicitly involve an analysis of a long time parametrix as suggested by \eqref{fourierreslog}.

As noted above, the rigorous derivations of the Van Vleck formula typically proceed by justifying an approximation to $K_t(x,y)$ by microlocal methods, then applying stationary phase to the oscillatory integrals which result.  The catch is that justifying these approximations come at a great technical expense if one wishes to obtain uniform approximations up to the Ehrenfest time.  Early work in this direction involved the Laplacian on a compact Riemannian manifold. The work \cite{Berard77} concerns manifolds without conjugate points/caustics and shows that after lifting the problem to the universal cover, the Hadamard parametrix is effective to this end.  The later work \cite{Volovoy} operates in a much more general setting and instead shows that the global theory of FIOs allows for effective approximations up to the Ehrenfest time.  But this involved delicate estimates on numerous parametrizations of the underlying Lagrangian submanifolds.

In the present work, we work with a different parametrix often referred in the physics literature as the \emph{Herman-Kluk} approximation, which was rigorously justified up to Ehrenfest time scales by Swart and Rousse in \cite{SwartRousse} and later by Robert in \cite{RobertHK}.  In what follows, we cite results from the latter work as it gives precise bounds on the amplitude and its derivatives. The main idea is to define the complex phase
\begin{equation}\label{phaseintro}
	\Phi(t,x,y,q,p) = S(t,q,p) + p_t(q,p)\cdot(x-q_t(q,p)) - p\cdot(y-q) + \frac i2 \big(|x-q_t(q,p)|^2 + |y-q|^2 \big),
\end{equation}
then find an amplitude $a(t,q,p)$ so that 
\begin{equation}\label{HKintro}
	K_t(x,y) \approx \frac{1}{(2\pi\h)^{3d/2}}\int_{T^*\RR^d} e^{\frac i\h \Phi(t,x,y,q,p)}a(t,q,p) \,dqdp.
\end{equation}
In this sense, the Herman-Kluk approximation can be viewed as a superposition of Gaussians whose centers have been propagated along the classical motion.  In the chemical physics literature these are sometimes referred to as ``frozen'' Gaussian approximations since the profile of the Gaussian is unchanged.  This is slightly different from the approach in \cite{BilyRobert}, which uses ``thawed'' Gaussian approximations that amount to replacing the term $|x-q_t(q,p)|^2$ by 
\begin{equation*}
	\Gamma_t(q,p) (x-q_t(q,p))\cdot (x-q_t(q,p)), \text{ with } \Im \Gamma_t >0,
\end{equation*}
and $\Gamma_t$ solves a Riccati equation, allowing the profile of the Gaussian to evolve with time.  

As with \cite{LaptevSigal}, \cite{BilyRobert}, one of the virtues of Herman-Kluk integral \eqref{HKintro} is that it only involves one oscillatory integral, rather than a sum of them.  Consequently, obtaining an approximation of Van Vleck type just amounts to applying the principle of stationary phase (for complex phases) near critical points of the phase in \eqref{phaseintro}.  But the caveat here is that if $|t|$ is allowed to vary in an $\h$-dependent interval, as opposed to a fixed interval, we obtain a phase which essentially depends on $\h$ itself.  To handle this, we arrive at the second topic mentioned in the title.

\subsection{Stationary phase asymptotics for frequency-dependent phases} Consider the $\h\to 0$ asymptotics of oscillatory integrals $I_\h$ defined by $C^\infty$, $\h$-dependent phases and amplitudes $\phi_\h,u_\h$:
\begin{equation}\label{initialIdef}
  I_\h := \int_{\RR^{\d}} e^{\frac{i}{\h}\phi_\h(x)}u_\h(x)\,dx. 
\end{equation}
We start with the typical hypotheses that $\phi_\h, u_\h$ are complex valued and
\begin{equation}\label{minimalsphyp}
\frac{\prtl \phi_\h}{\prtl x}(0)=0, \text{ and } \Im \phi_\h(x) \geq 0 \text{ on }\supp(u_\h) \text{ with } \Im \phi_\h(0)=0.
\end{equation}
Next, we introduce parameters $\mu,\nu,\sigma \geq 0$ which are assumed to control the growth of the $C^\infty$ phases and amplitudes.  We assume that the parameters $\mu,\nu$ control the growth of the derivatives of the phase, and $\sigma$ controls the growth of inverse of the Hessian in that
\begin{equation}\label{phihyp}
\begin{split}
|\prtl^\alpha_x \phi_\h(x)| &\lesssim_{\alpha}
      \h^{-\mu-\nu(|\alpha|-2)} \text{ for } |\alpha|\geq 2,
\\
    &\Big\|\Big( \frac{\prtl^2 \phi_\h}{\prtl x^2}(0)\Big)^{-1}\Big\| \lesssim \h^{-\sigma},
\end{split}
\end{equation}
Moreover, we assume $\mu,\nu,\sigma$ satisfy
\begin{equation}\label{asymptotichyp}
	1>5\mu+6\sigma+2\nu.
\end{equation}
From these three parameters, we define
\begin{equation}\label{rhodef}
	\rho = \sigma+\nu+\mu
\end{equation}
We then assume that $\rho$ controls the growth of the derivatives of the amplitude in that
\begin{equation}\label{amplitudehyp}
  |\prtl^\alpha_x  u_\h(x)|   \lesssim_\alpha \h^{-\rho|\alpha| }, \text{ and }
  \supp(u_\h) \subset \big\{x :|x|\leq \veps \h^\rho \big\} \text{ for some $\veps>0$ sufficiently small}.
\end{equation}
Under the hypotheses above, we obtain the following uniform estimates on the asymptotic expansion of the integrals $I_\h$, which will be proved in \S\ref{S:sp} below.
\begin{theorem}\label{T:spthm}
	Suppose the phase and amplitude of $I_\h$ in \eqref{initialIdef} satisfy \eqref{minimalsphyp}, \eqref{phihyp},  \eqref{amplitudehyp}, for some parameters $\rho,\sigma,\nu,\mu$ satisfying \eqref{asymptotichyp}, \eqref{rhodef}.  Then for each integer $N \geq 1$,
\begin{equation}\label{spconclusion}
  \left| I_\h - (2\pi\h)^{\frac{\d}{2}}
  \frac{e^{\frac i\h\phi_\h(0)}}{\det\!^{1/2}\left(\frac 1i\frac{\prtl^2\phi_\h}{\prtl x^2}(0)\right) }
  \sum_{j=0}^{N-1}\h^j \Lsc_{j,h}u_\h \big|_{x=0} \right|\\
  \lesssim_k \left|\det\!^{-1/2}\mbox{$\left(\frac 1i\frac{\prtl^2\phi_\h}{\prtl x^2}(0)\right)$}\right|\h^{(1-\sigma-2\rho)N+\frac \d2},
\end{equation}
where the $\Lsc_{j,\h}$ are the differential operators  defined by
\begin{gather}
  \Lsc_{j,\h} f :=
  \sum_{\ell-m=j}\sum_{2\ell \geq 3m} \frac{i^{-j}2^{-\ell}}{\ell!m!}
  \Big\langle\mbox{$\left(\frac{\prtl^2\phi_\h}{\prtl x^2}(0)\right)^{-1}$} D_x, D_x\Big\rangle^{\ell}\big(g_\h^m f\big) \label{Ljdef},\\
   g_\h(x) := \phi_\h(x)-\phi_\h(0)-\Big(\frac{\prtl^2\phi_\h}{\prtl x^2}(0) \Big) x \cdot x.
\label{gdef}
\end{gather}
For each $k$, the implicit constant in \eqref{spconclusion} depends only on finitely many of the implicit constants appearing in \eqref{phihyp}, \eqref{amplitudehyp}.  Moreover, the hypothesis give that $ |\Lsc_{j,h}u_\h |_{x=0}| \lesssim_j \h^{-j(\sigma+2\rho)}$.
\end{theorem}

If the phases $\phi_\h$ above satisfied bounds that were uniform in $\h$, i.e. if one took $\sigma = \mu=\nu =0$ in \eqref{phihyp}, then asymptotic expansion in \eqref{spconclusion} would follow by the usual proof of stationary phase asymptotics \emph{mutatis mutandis}.  Moreover, in such cases, one could allow for $\h$-dependent amplitudes whose derivatives may grow with $\h$ as in \eqref{amplitudehyp}. Indeed, results such as \cite[Theorem 7.7.5]{HormanderI} show that the errors in the asymptotic expansion can be controlled by the $C^{2k}$ norm of the amplitude.   Hence the novelty of the present result is that we allow for \emph{phases} whose derivatives grow with $\h$ and whose inverse Hessian may also grow with $\h$.  

Theorem \ref{T:spthm} is a natural extension of the stationary phase asymptotics, but we are unaware of any works that justify the full asymptotic expansion in our setting.  However, there has been recent interest in formulating more robust approaches to the $|I_\h| = O(\h^{{d/2}}|\det( \frac{\prtl^2\phi_\h}{\prtl x^2}(0))|^{-1/2})$ bound that results from \eqref{spconclusion}.  A work of Tacy \cite{TacyStatPhase} generalizes this bound for cases where the amplitude is more singular than \eqref{amplitudehyp}, enough so that asymptotics are not possible, and the growth of the Hessian akin to \eqref{phihyp} is allowed.  A work of Alazard, Burq, and Zuily \cite{AlazardBurqZuilyStatPhase}, give this bound in a fashion which reflects its dependence on the derivatives of the phase and amplitude in a precise way.  The more recent work of Oh and Lee \cite{OhLeeUniform} reexamines the bounds in both these works, allowing for even weaker hypotheses on the phase function.  

To prove Theorem \ref{T:spthm}, we first perform an $\h$-dependent dilation of variables in the integrals $I_\h$ in \eqref{initialIdef} so that the derivatives of the phase function are uniformly bounded.  We then use H\"ormander's approach to stationary phase in \cite[\S7.7]{HormanderI} to obtain bounds on the asymptotic expansions of the new oscillatory integrals.  

It is not clear if the condition in \eqref{asymptotichyp} is optimal, it simply arises in the approach we take here.  At first glance, the constraint on the support of the amplitude in \eqref{amplitudehyp} may appear to limit the value of the theorem, but we shall see that in our applications to the Van Vleck formula, we can obtain lower bounds on the imaginary part of the phase away from such small neighborhoods.

It would be interesting to extend Theorem \ref{T:spthm} to cases where $\phi_\h,u_\h$ depend smoothly on a second variable $y$, i.e. establish a stationary phase asymptotics with a parameter $y$.  In this case, if the critical locus $\{ \frac{\prtl \phi_\h}{\prtl x}(x,y) =0\}$ is assumed to lie in $\RR$ (necessarily the case when $\phi_\h$ is real-valued), then the asymptotics just follow from a simple translation of coordinates in the integrals $I_\h$.  Moreover, the formula \eqref{Ljdef} shows the dependence of the coefficients on the phase in an explicit manner, which in turn allows for an analysis of their regularity in $y$.  But for general complex phases, these matters, including what even defines the critical locus, are a little more subtle.

\subsection{The Van Vleck formula on Ehrenfest time scales}\label{SS:VVIntro}
In this section, we give a precise statement of our main result on the Van Vleck formula, followed by some important remarks.

We begin by detailing our hypotheses on the classical Hamiltonian $H(t,z)$. The first of these is drawn from Robert's work \cite{RobertHK} so that we may apply his rigorous version of the Herman-Kluk approximation.  We borrow the definition there that a symbol $b(z)$ on $\RR^{2d}$ is said to lie in $\Ocal_m$ if for each $k \geq m$, we have $ \|b\|_{\dot{C}_z^k}  < \infty$ (see \S\ref{SS:Notation} for notation such as $\dot{C}_z^k$).

\begin{enumerate}[label=\textbf{A\arabic*}]
	\item\label{HamHyp} \textbf{Assumptions on the Hamiltonian.} We allow the Hamiltonian $H$ to depend on $\h$, assuming it admits an asymptotic expansion 
	$
	H(t,z) \sim \sum_{j=0}^{\infty} \h^j H_j(t,z),	
	$
	where each $H_j(t,\cdot)$ lies in $\Ocal_{(2-j)_+}$ with
	\begin{equation}\label{Hjbds}
		\|H_j(t,\cdot)\|_{\dot{C}^k_z} \lesssim_k 1, \quad \text{ for } k \geq (2-j)_+ .
	\end{equation}
	It is then further assumed that if 
	$
		R_N(t,z):= \h^{-(N+1)}\Big(H(t,z) - \sum_{j=0}^N \h^j H_j(t,z)\Big) ,
	$
	then $R_N(t,\cdot)$ lies in $ \Ocal_0$ for $N \geq 1$ with 
	\begin{equation}\label{Rbds}
		\sup_{\h \in (0,1], t \in \RR}\| R_N(t,\cdot)\|_{\dot{C}_z^k} \lesssim_k 1 . .
	\end{equation}
	The hypotheses imply that the full symbol $H(t,\cdot)$ lies in $\Ocal_2(2d)$, that is, it is a \emph{subquadratic} symbol.  It is known that these hypotheses are sufficient to ensure the existence of $\U_t$.
	
	\item\label{FlowHyp} \textbf{Global assumptions on the flow.} We assume that the leading order symbol $H_0(t,z)$ generates a complete Hamiltonian flow $\kappa_t$ on $ T^*\RR^d$.  We further assume there is a continuous, even function $\omega(t) \geq 1$ which is nondecreasing for $t >0$ and satisfies both $0<\inf_{t> 0}\frac{\omega(t)}{t}$, 
\begin{equation}\label{flowbds}
	|\prtl_z^\gamma \kappa_t(z)| \lesssim_{\gamma} \omega(t)^{|\gamma|}, \qquad |\gamma| \geq 1.
\end{equation}
In the work of Bouzouina and Robert \cite[Lemma 2.2]{BouzouinaRobert}, it is shown that the subquadratic hypothesis on $H_0$ means that at the very least, this bound is satisfied with $\omega(t) = \exp(\Gamma|t|)$ where $\Gamma = \|H_0\|_{\dot{C}^2}$.

	\item\label{WHyp} \textbf{Assumptions on the semiclassical Fourier multiplier $\widehat{\Theta}$.} 
	Let $\Theta$ be a smooth, compactly supported function which is independent of $\h$.
	Our results will ultimately give asymptotics on compositions $\U_t \circ \widehat{\Theta}$ where $\widehat{\Theta}$ is the Fourier multiplier with symbol defined by the semiclassical $\h$-Fourier transform whose kernel is given by \eqref{Thetakernel} below.  In other words, $\widehat{\Theta}$ is the $\h$-Weyl quantization of a symbol $\Theta(p)$ which is independent of $q$.

	\item\label{PtsHyp} \textbf{Assumptions on the flow at points $y \in \RR^d$.} Given a pseudodifferential operator $\widehat{\Theta}$ as in \ref{WHyp}, we then make a pointwise hypothesis at some $y \in \RR^d$
\begin{equation}\label{nonsinghypmain}
	\sup \Big\{ \Big\| \Big(\frac{\prtl q_t}{\prtl p}(y,p)\Big)^{-1} \Big\|  : p \in \supp(\Theta) + B(0,\h^\rho) \Big\}\lesssim \h^{-\delta}.
\end{equation}
	Here $\rho,\delta>0$ are parameters whose role will be clarified in the main theorem.  Given $y$ satisfying these hypotheses, and another point $x \in \RR^d$, we introduce the notation 
	\begin{equation*}
		 \widetilde{\Xi}_t (x,y) = \{p \in \supp(\Theta) + B(0,\h^\rho) : q_t(y,p) =x\}.
	\end{equation*}
\end{enumerate}

Under these hypotheses we shall obtain the following Van Vleck approximation, proved in \S\ref{S:VV}.

\begin{theorem}\label{T:VVthm}
	Suppose the Hamiltonian $H$ satisfies the conditions \ref{HamHyp}, its flow satisfies \ref{FlowHyp}, and that $\widehat{\Theta}$ satisfies \ref{WHyp}.   Assume further that $t \neq 0$ and $\omega(t) \leq \h^{-\lambda}$ for some $\lambda >0$. Finally, suppose that $\delta\geq 0$ is such that 
	\begin{equation*}
		6\delta + 24\lambda <1,
	\end{equation*}
	and that \ref{PtsHyp} is satisfied with this value of $\delta$ and $\rho = \delta + 5\lambda$.  There exists an operator $\V_t$, with kernel $L_t$ admitting an asymptotic expansion, such that $\|\U_t \circ \widehat{\Theta} - \V_t\|_{L^2\to L^2} = O(\h^\infty)$, that is, 
	\begin{equation*}
		\|\U_t \circ \widehat{\Theta} - \V_t\|_{L^2\to L^2} \lesssim_N \h^N \quad \text{ for any }N.
	\end{equation*}
	If $\widetilde{\Xi}_t (x,y) = \emptyset$, then $|L_t(x,y)| \leq_N C_N \h^{N}$ for any $N$.  Otherwise, if $\widetilde{\Xi}_t (x,y) \neq \emptyset$, then for each $\eta \in \widetilde{\Xi}_t(x,y)$ there exists a sequence $\{ b_{\h,k}(y,\eta)\}_{k=0}^\infty$ so that the kernel $L_t(x,y)$ satisfies for $N \in \mathbb{N}$
	\begin{multline}\label{VVmaintheorem}
		\Big|L_t(x,y)-\frac{1}{(2\pi \h)^{\frac d2}}\sum_{ \eta \in \widetilde{\Xi}_t (x,y)}e^{\frac i\h S_\eta-i\frac{\pi}{2} \theta_\eta} 
		\Big|\det \frac{\prtl^2 S_\eta}{\prtl x \prtl y}\Big|^{1/2} 
		\sum_{k=0}^{N-1} \h^k b_{\h,k}(y,\eta) \Big|\\
		 \leq C_{\lambda,\delta,N} \h^{N(1-3\delta -12\lambda) -\frac{\d}{2}} \big(\# \widetilde{\Xi}_t (x,y)\big)\Big( \max_{\eta \in \widetilde{\Xi}_t (x,y)}\Big|\det \frac{\prtl^2 S_\eta}{\prtl x \prtl y}\Big|^{1/2} \Big).
	\end{multline}
	Here $S_\eta= S(t,y,\eta)$ where $S$ is as in \eqref{actionintro}, and $\theta_\eta \in \frac 12 \mathbb{Z}$ is the Maslov index determined by the path. 	The hypotheses imply that the sequence $\{b_{\h,k}\}_{k=0}^\infty$ satisfies 
	\begin{equation}\label{vvcoeffbds}
		|b_{\h,k}(y,\eta)| \lesssim_k \h^{-k(3\delta + 12 \lambda)}.
	\end{equation}
\end{theorem}

\noindent\textbf{Remarks.}
\begin{enumerate}[label=\textbf{\arabic*}.]
	\item We emphasize that the constant $C_{\lambda,\delta,N} $  in \eqref{VVmaintheorem} of course depends on $\lambda,\delta,N$ and implicitly the Hamiltonian as well, namely the implicit constants in \eqref{Hjbds}, \eqref{Rbds}, \eqref{nonsinghypmain}.  Otherwise, the bound is completely uniform for any choice of $t \neq 0$.  In particular, the approximation holds even if the hypothesis \eqref{nonsinghypmain} is violated at some time $t_1 \neq t$, so long as it holds at $t$.
	
	\item\label{R:coeffs} The coefficients $b_{\h,k}(y,\eta)$ here are allowed to depend on $\h$, a consequence of our liberal use of the Borel lemma. This means they are not uniquely determined: indeed, if $c_{\h,k}(y,\eta)$ satisfies $|c_{\h,k}(y,\eta)| \lesssim \h^{1-(k+1)(3\delta+12\lambda)}$, then $b_{\h,k},b_{\h,k+1}$ can be replaced by $b_{\h,k}+c_{\h,k},b_{\h,k+1} - \h^{-1}c_{\h,k}$ respectively (modulo the implicit constant in \eqref{vvcoeffbds}).
	However, our approach shows that
	\begin{equation}\label{vvleading}
		\Big| b_{\h,0}(y,\eta) - \Theta(\eta)\exp\big(-i\int_0^tH_1(s,\kappa_s(y,\eta))\,ds\big)\Big| \lesssim \h^{1-3\delta-12\lambda },
	\end{equation} 
	which gives a natural way to express the leading order coefficient in the approximation. 
	
	\item The hypotheses $6 \delta + 24 \lambda <1$ and $\rho = \delta + 5\lambda$ are admittedly technical hypotheses that are imposed so that the conditions \eqref{asymptotichyp}, \eqref{rhodef} are satisfied in Theorem \ref{T:spthm}.  To appreciate why this still yields meaningful results, it is helpful to reference the case where $\omega(t) = \exp(\Gamma|t|)$, so that the assumption $\omega(t) \leq \h^{-\lambda}$ is equivalent to saying that $|t| \leq \frac{\lambda}{\Gamma}|\log\h|$.  If it is  further assumed that $\delta = 0$, then \eqref{VVmaintheorem} gives uniform asymptotics when $|t| \leq \frac{1-\veps_0}{24\Gamma}|\log\h|$ for some $\veps_0>0$.  When $|t| \geq 1$, the bounds in \eqref{nonsinghypmain} with $\delta = 0$  are known to be satisfied when $H_0(q,p) = -\frac 12 \sum_{ij} g^{ij}(q)p_ip_j$ is a Hamiltonian generating a geodesic flow of a Riemannian metric with nonpositive sectional curvatures and follow by Jacobi field comparison estimates.  More generally if the metric has no conjugate points, then such bounds follow for some $\delta>0$ as a consequence of works of Green \cite{GreenConjugatePoints} (treating $d=2$ only) and Bonthonneau \cite{BonthonneauTheta}.  
	
	\item Arguably the presence of the Fourier multiplier $\widehat{\Theta}$ limits the generality of Theorem \ref{T:VVthm} since it does restrict to momenta contained in a compact set.  However, for many Schr\"odinger operators of interest, the sets $\Xi_t(x,y) = \{ p: q_t(y,p) =x \}$ may naturally be contained in a compact set anyway.  For example, in \cite{BilyRobert} it is observed that if $H(t,z) = \frac 12 |p|^2 + V(t,q)$ where $|\prtl_q V(t,q)| \lesssim (1+|q|)^{1-\veps}$, then $\Xi_t(x,y)$ is indeed contained in a compact set since Hamiltonian rays escape to infinity.  On the other hand, including $\hat \Theta$ in the calculation may be significant in applications.
	
	\item It would be interesting to investigate the extent to which the hypothesis in \eqref{nonsinghypmain} can be relaxed.  Even for problems on uniformly bounded time intervals, understanding the asymptotics of $L_t(x,y)$ at points where $x=q_t(y,\eta)$ and $\frac{\prtl q_t}{\prtl \eta}(y,\eta)$ drops rank is already subtle.  On the Ehrenfest time scales considered here, there appears to be the additional challenge of coping with the contributions of cases where we may have $x \neq q_t(y,\eta)$, but $|x-q_t(y,\eta)| \lesssim \h^{\alpha}$ for some $ \alpha >\frac 12$ so that $\eta$ may influence the asymptotics even though formally it does not lie in $\Xi_t(x,y)$.  Nonetheless, we believe that the techniques further developed here may be illuminating in such problems.
	
	\item As noted above, in this work we use the subquadratic hypothesis of Robert \cite{RobertHK} detailed in assumption \ref{HamHyp}.  This allows for us to appeal to his bounds \eqref{ajbdslongtime} below on the amplitudes appearing in the Herman-Kluk approximation, minimizing the technicalities in the present work.  However, we expect that (variations on) the approach here is applicable to larger classes of Hamiltonians.  In particular, the aforementioned results on propagation of coherent states \cite{CombescureRobertWavePackets} should apply more broadly, and hence it may be possible to combine the ``thawed Gaussian" approach of Bily and Robert with the stationary phase asymptotics here to relax the subquadratic hypothesis.
	
	\item While Theorem \ref{T:VVthm} makes no assumption as to whether $|t|$ is small or large, the hypothesis \eqref{nonsinghypmain} is often violated for small times $|t| \ll \h^{\delta}$ as it is common that $\|(\frac{\prtl q_t}{\prtl p})^{-1}\| \approx t^{-1}$. However, the asymptotics for small times are already known anyway and typically can be approached by other means.
\end{enumerate}

\subsection{Notation}\label{SS:Notation}
The notation $A \lesssim B$ means that $A \leq CB$ for some uniform constant $C$, which may depend on $H$, but is always taken to be independent of $\h$.  When the implicit constant $C$ depends on other quantities of significance, they are typically included in the subscript of $\lesssim$.  We routinely omit dependence on $\h$ in our notation, with the understanding that it is an implicit parameter.  However, it is sometimes significant for us to exhibit the dependence of certain norms taken with respect to a certain variable.  For example, given a $C^k$ function $f(y)$, we use the notation
\begin{equation*}
	\|f\|_{\dot{C}_y^k} =  \sum_{|\alpha| = k} \sup_{y \in \RR^d} |\prtl_y^\alpha f(y)|, \text{ and } \|f\|_{C_y^k} =   \sum_{j=0}^k  	\|f\|_{\dot{C}_y^j},
\end{equation*}
to define homogeneous and inhomogeneous $C^k$ seminorms defined with respect to $y$ (as opposed to some other variable).  When no subscript is given in a norm, it should be interpreted as a matrix norm. We also use the common conventions that $D_{y_j} = \frac 1i \prtl_{y_j}$ and $x_+ = \frac 12 (x+|x|)$ when $x \in \RR$.

\subsection*{Acknowledgements} The author is grateful to Didier Robert and Peter Miller for helpful comments on this work.  He was supported in part by the National Science Foundation grant DMS-1565436.

\section{Stationary phase asymptotics for frequency dependent phases}\label{S:sp}
In this section, we prove Theorem \ref{T:spthm}.  We will drop the $\h$ in the notation for $\phi, u, g$ for convenience. Without loss of generality, we also assume 
\begin{equation*}
\phi(0) =0.
\end{equation*}
We begin by recalling two theorems in \cite{HormanderI}, with modest changes to their statements.
\begin{theorem}[H\"ormander, Theorem 7.7.1 in \cite{HormanderI}]\label{T:HoNSP} Let $\mathcal{K} \subset \RR^d$ be a compact set, $\mathcal{X}$ an open neighborhood of $\mathcal{K}$ and $k$ a nonnegative integer.  If $u \in C_0^k (\mathcal{K})$, $\phi \in C^{k+1}(\mathcal{X})$, and $\Im \phi \geq 0$ in $\mathcal{X}$, then for $0< \delta<1$ 
\begin{equation}\label{HoNSP}
\Big| \int_{\RR^{\d}} e^{\frac{i}{\delta}\Phi (y)}u (y)\,dy \Big| \leq C_k \delta^k \sum_{|\alpha| \leq k} \sup |\prtl_y^\alpha u |\Big| \frac{\prtl \Phi }{\prtl y}\Big|^{|\alpha|-2k}
\end{equation}
Here $C_k$ is bounded when $\Phi $ stays in a bounded set in $C^{k+1}(\mathcal{X})$.
\end{theorem}
\begin{theorem}[H\"ormander, Lemma 7.7.3 in \cite{HormanderI}]\label{T:HoQuad} Let $A$ be a symmetric, non-degenerate matrix with $\Im A \geq 0$.  Then  for every integer $k>0$  and every Schwartz class function $u$ on $\RR^d$,
\begin{multline}\label{hoquadL2}
\Big| \int_{\RR^{\d}} e^{\frac{i}{2\delta}Ay \cdot y}u (y)\,dy - \frac{(2\pi \delta)^{\frac d2}}{\det\!^{1/2}(\frac 1i A)} \sum_{j=0}^{k-1} \frac{1}{j!}\Big(\frac{\delta}{2i}\Big)^j \langle A^{-1} D_y,D_y \rangle^j u(0)\Big| \\
\leq C_{k,d}\, \delta^{\frac d2 + k} \|A^{-1}\|^{\frac d2 + k} \sum_{|\alpha| \leq 2k+\lceil \frac{d+1}{2} \rceil} \|\prtl^\alpha u\|_{L_y^2}.
\end{multline}
\end{theorem}
A crucial feature of Theorem \ref{T:HoNSP} for us is last line, emphasizing that the nature of H\"ormander's proof means that the constant $C_k$ can be taken as uniform as long as $\Phi$ lies in a bounded set in $C^{k+1}(\mathcal{X})$.   Indeed, since the proof relies on a standard integration by parts involving the phase, the constant is independent of the amplitude $u$. The theorem will be applied in cases where $u$ vanishes to sufficiently high order at 0 to cancel out the singularity of $|\frac{\prtl \Phi }{\prtl y}|^{|\alpha|-2k}$ there.  Similarly, a key feature of Theorem \ref{T:HoQuad} is that the constant $C_{k,d}$ on the right hand side of \eqref{hoquadL2} depends only on $k,d$, meaning it will suffice to estimate the contribution of $\|A^{-1}\|$.  

Below, we will consider $C^\infty$ functions $u$ with support contained in a common compact set $\mathcal{K}$.  In this case, H\"older's inequality and \eqref{hoquadL2} imply the following bound, where the right hand side involves $C^j$ spaces instead of $L^2$
\begin{multline}\label{hoquadCk}
\Big| \int_{\RR^{\d}} e^{\frac{i}{2\delta}Ay \cdot y}u (y)\,dy - \frac{(2\pi \delta)^{\frac d2}}{\det\!^{1/2}(\frac 1i A)} \sum_{j=0}^{k-1} \frac{1}{j!}\Big(\frac{\delta}{2i}\Big)^j \langle A^{-1} D_y,D_y \rangle^j u(0)\Big| \\
\lesssim_{\mathcal{K},k,d}
 \delta^{\frac d2 + k} \|A^{-1}\|^{\frac d2 + k} \| u\|_{C^{2k+\lceil \frac{d+1}{2} \rceil }_y}.
\end{multline}

\begin{proof}[Proof of Theorem \ref{T:spthm}]
Once again, we stress that our approach to stationary phase relies on the treatment given by H\"ormander in \cite[\S7.7]{HormanderI}.  The only new contribution here is to observe the effect of dilations on the integrals.
	
We begin by observing that if $\Lsc_{j,h}$ defined as in \eqref{Ljdef}, \eqref{gdef}, then the hypotheses of Theorem \ref{T:spthm} imply that
\begin{equation}\label{sptermsbd}
	|\Lsc_{j,h}u_\h |_{x=0}| \lesssim_j \h^{-j(\sigma+2\rho)},
\end{equation}
which is the very last claim in that theorem.  A Taylor expansion shows that $g_\h(x) = \sum_{|\gamma|=3} x^\gamma r_\gamma(x)$, where $r_\gamma$ is determined by the third order partials of $\phi$ and hence satisfies $|\prtl^\alpha r_{\gamma} | \lesssim_\alpha \h^{-(\mu+\nu)-\nu|\alpha|}$ by \eqref{phihyp}.  We then have the following bound on derivatives of an $m$-fold product of the $r_\gamma$'s, which can be seen by induction on $m$,
\begin{equation*}
	\left|\prtl^\alpha\Big( \prod_{j=1}^{m}r_{\gamma_j} \Big)\right| \lesssim_{\alpha,m} \h^{- (\mu+\nu)m -\nu|\alpha|}.
\end{equation*}
Thus if $\ell-m=j$, $2\ell \geq 3m$, we have the following bound on the corresponding term in $\Lsc_{j,h}u_\h |_{x=0}$
\begin{equation}\label{lmbd}
	\left| \Big\langle\mbox{$\left(\frac{\prtl^2\phi_\h}{\prtl x^2}(0)\right)^{-1}$} D_x, D_x\Big\rangle^{\ell}\big(g_\h^m u \big) \Big|_{x=0} \right|\lesssim \h^{-\ell \sigma} \sum_{|\alpha| = 2\ell-3m}\big |\prtl^\alpha(r_m u)|_{x=0}\big|,
\end{equation}
which uses that $g^m$ has a zero of order $3m$ and the inverse Hessian bound in \eqref{phihyp}. The Leibniz rule then shows that each term in the sum on the right satisfies
\begin{equation*}
	\big|\prtl^\alpha(r_m u)|_{x=0}\big| \lesssim_j \sum_{\alpha_1+\alpha_2 = \alpha} \h^{-(\mu+\nu)m - \nu|\alpha_1| -\rho|\alpha_2|} \lesssim_j \h^{-(\mu+\nu)m - \rho|\alpha|} = \h^{-(\mu+\nu)m - \rho(2\ell-3m)}, 
\end{equation*}
where the last bound uses that $\rho\geq \nu$.  The bound \eqref{sptermsbd} now follows since $\rho = \mu+ \nu+ \sigma$ (cf. \eqref{rhodef}), the right hand side of \eqref{lmbd} is thus seen to be dominated by
\begin{equation*}
	 \h^{-\sigma \ell-(\mu+\nu)m - \rho(2\ell-3m)} = \h^{-j(\sigma+2\rho)}.
\end{equation*}

The next observation is that it suffices to show that for each integer $k\geq 1$,
	\begin{equation}\label{spreduction}
		\left| I_\h - (2\pi\h)^{\frac{\d}{2}}
		\det\!^{-1/2}\left(\frac 1i\frac{\prtl^2\phi_\h}{\prtl x^2}(0)\right) 
		\sum_{j=0}^{k-1}\h^j \Lsc_{j,h}u_\h \big|_{x=0} \right|\\
		\lesssim_k \h^{(1-5\mu-6\sigma -2\nu )k+\rho\d}.
	\end{equation}
To see this, first note that the assumptions in \eqref{phihyp} give the lower bound
$$
\h^{\frac{\d\mu}2 } \lesssim \left|\det\!^{-1/2}\mbox{$\left(\frac 1i\frac{\prtl^2\phi_\h}{\prtl x^2}(0)\right)$}\right|
$$
Consequently, to see \eqref{spconclusion} for some integer $N \geq 1$, take $k\geq N$ large enough so that the right hand side of \eqref{spreduction} is bounded above by the right hand side of \eqref{spconclusion}.  The bounds \eqref{sptermsbd} then imply
\begin{equation*}
	\left| (2\pi\h)^{\frac{\d}{2}}
	\det\!^{-1/2}\left(\frac 1i\frac{\prtl^2\phi_\h}{\prtl x^2}(0)\right) 
	\sum_{j=N}^{k-1}\h^j \Lsc_{j,h}u_\h \big|_{x=0} \right| \lesssim \left|\det\!^{-1/2}\mbox{$\left(\frac 1i\frac{\prtl^2\phi_\h}{\prtl x^2}(0)\right)$}\right| \h^{N(1-\sigma-2\rho)+ \frac \d2},
\end{equation*}
at which point \eqref{spconclusion} follows.
	
	We now make the change of variables $y = \h^{-\rho} x$ in the integral $I_\h$ in \eqref{initialIdef}. Setting 
\begin{equation*}
	\delta = \h^{1+\mu-2\rho} ,
\end{equation*}	
we are led to the following oscillatory integral, which is identical to $\h^{-\rho d}I_\h$:
\begin{equation}\label{yintegral}
 \int_{\RR^{\d}} e^{\frac{i}{\delta}\Phi(y)}u(y)\,dy, \qquad \Phi(y) := \h^{\mu-2\rho}\phi(y).
\end{equation}
Given \eqref{phihyp} and the relation $\h^{\rho}\frac{\prtl}{\prtl x_j} = \frac{\prtl}{\prtl y_j} $ from the chain rule, the $k$-th derivatives of $\Phi$ with respect to $y$ satisfy
\begin{equation}\label{bigphihyp}
\|\Phi\|_{\dot{C}^k_y} \lesssim_k \h^{(\rho-\nu)(k-2)_+}, \qquad k \geq 0.
\end{equation}
Indeed, for $k \geq 2$, this just follows from \eqref{phihyp}, the definition of $\Phi$, and the chain rule.  For $k=0,1$ this just follows by using a Taylor expansion to bootstrap the bound in the $k= 2$ case, in light of our assumptions $\Phi(0)=0$, $\frac{\prtl \Phi}{\prtl y}(0) =0$.
Moreover, we have the relation
\begin{equation}\label{hessianchange}
\frac{\prtl^2 \Phi}{\prtl y^2}(y) = \h^{\mu}\frac{\prtl^2 \Phi}{\prtl x^2}(x).
\end{equation}
We further observe that with $\veps$ as in \eqref{amplitudehyp},
\begin{equation}\label{ubounds}
\supp(u) \subset \{y: |y| \leq \veps\} \quad \text{and } \quad \| u\|_{\dot{C}^k_y} \lesssim_k 1, \qquad k \geq 0.
\end{equation}

As in \cite[Theorem 7.7.5]{HormanderI}, define $G$ as the remainder in the quadratic Taylor expansion of $\Phi$,
\begin{equation}\label{GAdef}
G (y): = \Phi(y) - \frac 12 Ay\cdot y, \qquad \text{ where }A: = \frac{\prtl^2 \Phi}{\prtl y^2}(0) = \h^{\mu}\frac{\prtl^2 \Phi}{\prtl x^2}(0).
\end{equation}
We may also write $G$ as 
\begin{equation*} 
G(y) = \sum_{|\gamma| = 3} R_\gamma(y) y^\gamma, 
\qquad 
\text{ where }
\|R_\gamma\|_{\dot{C}_y^\ell } \lesssim_\ell \h^{(\ell+1)(\rho-\nu)}, \text{ for } \ell \geq 0.
\end{equation*}
The bounds on $R_\gamma$ follow from \eqref{bigphihyp} as they are controlled by the $\dot{C}^{3+\ell}$ norm of $\Phi$.
Induction on $m$ now gives the following bound on the product of the $R_\gamma$'s
\begin{equation}\label{rgammaprod}
\Big\|\prod_{j=1}^m R_{\gamma_j} \Big\|_{\dot{C}_y^\ell} \lesssim_{\ell ,m} \h^{(\ell+m)(\rho-\nu)} \quad \text{ for any }\ell \geq 0, m \geq 1.
\end{equation}

Asymptotics on \eqref{yintegral} now follow by the argument in \cite[Theorem 7.7.5]{HormanderI}. Define for $s \in [0,1]$ 
\begin{equation}\label{Jdef}
J(s) := \int_{\RR^{\d}} e^{\frac{i}{\delta} \Phi_s(y) }u(y)\,dy, \quad  \text{ where  }\Phi_s(y) := \frac 12 Ay \cdot y + sG(y).
\end{equation}
Hence $J(1)$ is equal to the integral in \eqref{yintegral} and any derivative of $J$ at $s=0$ is an oscillatory integral with purely quadratic phase, see \eqref{Jmexpress} below.  

Before proceeding, we observe some lower bounds on $\frac{\prtl \Phi_s}{\prtl y}$.
Note that \eqref{phihyp}, \eqref{hessianchange} give 
\begin{equation}\label{Ainvbds}
	\|A^{-1}\| \lesssim h^{-(\mu+\sigma)}.
\end{equation}
A Taylor expansion of $\frac{\prtl \Phi_s}{\prtl y} (y)$ implies that with $\veps$ as in \eqref{amplitudehyp}, we have for $|y| \leq \veps$,
\begin{equation*}
	\Big| \frac{\prtl \Phi_s}{\prtl y}(y) - Ay \Big| \lesssim \|\Phi\|_{\dot{C}_y^3}|y|^2 \lesssim \veps \h^{\rho-\nu}  |y| = \veps \h^{\mu+\sigma}  |y|,
\end{equation*}
since $\rho= \mu+\nu+\sigma$ by definition. Consequently, 
\begin{equation*}
	|y| \lesssim \h^{-(\mu+\sigma)} |Ay| \lesssim \h^{-(\mu+\sigma)} \Big| \frac{\prtl \Phi_s}{\prtl y}(y) \Big| + \veps  |y|,
\end{equation*}
 so by taking $\veps>0$ sufficiently small in \eqref{amplitudehyp}, we have
\begin{equation}\label{ycontrol}
	|y| \lesssim \h^{-(\mu+\sigma)} \Big| \frac{\prtl \Phi_s}{\prtl y}(y) \Big| .
\end{equation}

Returning to $J(s)$, a Taylor remainder estimate gives
\begin{equation}\label{JTaylor}
\left|J(1)  - \sum_{m=0}^{2k-1} \frac{1}{m!}J^{(m)} (0)\right| \leq \frac{1}{(2k)!}\sup_{0<s<1}|J^{(2k)}(s)|.
\end{equation}
Here $J^{(2k)}(s)$ is given by
\begin{equation*}
J^{(2k)}(s) = (-1)^k\delta^{-2k}\int e^{\frac{i}{\delta}(\frac 12 Ay \cdot y + sG(y))}[G(y)]^{2k}u(y)\,dy.
\end{equation*}
Hence \eqref{ubounds}, \eqref{rgammaprod}, \eqref{ycontrol} give for $|\alpha| \leq 3k$
\begin{equation*}
\big|\prtl^\alpha \big( [G(y)]^{2k} u(y)\big)\big| \lesssim_{k} \h^{2k(\rho-\nu)}|y|^{6k-|\alpha|} \lesssim_{k}  \h^{2k(\rho-\nu)-(6k-|\alpha|)(\mu+\sigma)}\Big| \frac{\prtl \Phi_s}{\prtl y} (y)\Big|^{6k-|\alpha|} .
\end{equation*}
Indeed, since $G^{2k}$ vanishes to order $6k$, the largest contributions to the left hand side here come from terms where all derivatives fall on the monomials $y^{\gamma}$ with $|\gamma| = 6k$, meaning the bound is dictated by the $\ell =0$ case of \eqref{rgammaprod}.  In this argument and below, it may be helpful to note that the bounds on $\Phi, G, R_\gamma$ show that the more times they are differentiated, the larger the power of $\h$ in the bound; hence the largest possible contributions typically result from considering the cases with fewest possible derivatives on these functions. 

Theorem \ref{T:HoNSP} with $k$ replaced by $3k$ now gives
\begin{equation}\label{J2kbound}
\begin{split}
|J^{(2k)}(s)| &\lesssim_k \delta^k \sum_{|\alpha| \leq 3k} \sup_y \big|\prtl^\alpha \big( [G(y)]^{2k} u(y)\big)\big| \Big| \frac{\prtl \Phi_s}{\prtl y} (y)\Big|^{|\alpha|-6k}\\
& \lesssim_k \h^{k(1+\mu-2\rho) + 2k(\rho-\nu)-6k(\mu+\sigma)} = \h^{k(1-5\mu-6\sigma -2\nu )} .
\end{split}
\end{equation}
Note that the largest possible contributions to the sum in the first line come from the $|\alpha| =0$ cases.

We now examine the terms in the Taylor expansion of $J$ in \eqref{JTaylor} 
\begin{equation}\label{Jmexpress}
J^{(m)}(0) = i^m \delta^{-m}\int e^{\frac{i}{2\delta}Ay \cdot y }[G(y)]^{m}u(y)\,dy,  \quad \text{ where } 0 \leq m \leq 2k-1.
\end{equation}
By \eqref{rgammaprod}, we have
\begin{equation*}
\|G^m u\|_{C_y^{2(k+m)+\lceil \frac{d+1}{2} \rceil}} \lesssim_{k,m}   \h^{m(\rho-\nu)},
\end{equation*}
since the largest contributions from the left hand side come from estimating terms of the form $G^m \prtl^{\alpha} u$ (for $|\alpha| \leq 2(k+m)+\lceil \frac{d+1}{2} \rceil$).
Now define
\begin{equation*}
E_{k,m} :=  J^{(m)}(0)- \frac{(2\pi \h)^{\frac d2}}{{\det\!^{1/2}(\frac 1i A ) }}\sum_{\ell =0}^{k+m-1} \frac{1}{2^{\ell}\ell !}(-i\delta)^{\ell-m}
\langle A^{-1} D_y,D_y \rangle^\ell [G(y)]^{m}u(y) \Big|_{y=0} ,
\end{equation*}
so that Theorem \ref{T:HoQuad}, with $k$ replaced by $k+m$, gives the following bound  
\begin{equation*}
| E_{k,m} | \lesssim _{k,m} \delta^{k+\frac d2} \|A^{-1}\|^{k+m+\frac d2} \|G^m u\|_{C_y^{2(k+m) + \lceil \frac{d+1}{2} \rceil}}  \lesssim _{k,m} \h^{(1-2\rho-\sigma)(k+\frac d2)}.
\end{equation*}
Note that the right hand side here is smaller than the upper bound on $|J^{2k}(s)|$ in \eqref{J2kbound}.

We have now shown that
\begin{equation}\label{yexpn}
\Big| \h^{-\rho d} I_\h - \frac{(2\pi \delta)^{\frac d2}}{{\det\!^{1/2}(\frac 1i A ) }}
\sum_{m=0}^{2k-1} \sum_{\ell =0}^{k+m-1} \frac{(-i\delta)^{\ell-m}}{2^{\ell}\ell ! m!}
\langle A^{-1} D_y,D_y \rangle^\ell  [G(y)]^{m}u(y)\Big|_{y=0} \Big| \lesssim \h^{k(1-5\mu-6\sigma -2\nu )}.
\end{equation}
Define the differential operators
\begin{equation*}
\Msc_{j}f:= (-i)^{j}\sum_{\ell-m = j} \sum_{2\ell \geq 3m} \frac{1}{2^{\ell}\ell ! m!}
\langle A^{-1} D_y,D_y \rangle^\ell (G^m f) .
\end{equation*}
This allows us to rewrite the sum on the left in \eqref{yexpn} as 
\begin{equation*}
\sum_{m=0}^{2k-1} \sum_{\ell =0}^{k+m-1} \frac{(-i\delta)^{\ell-m}}{2^{\ell}\ell ! m!}
\langle A^{-1} D_y,D_y \rangle^\ell [G(y)]^{m}u(y) \Big|_{y=0}=\sum_{j=0}^{2k-1} \delta^j \Msc_{j}u  \Big|_{y=0}
\end{equation*}
Indeed, since $G^m$ vanishes to order $3m$, any term on the left hand side with $2\ell < 3m$ vanishes, so this is just a rearrangement of the nonzero terms in that sum.  We further observe that the functions $G,g$ defined in \eqref{GAdef}, \eqref{gdef} respectively satisfy $G(y) = \h^{\mu-2\rho} g(x) $.   Combining this with \eqref{hessianchange} and the relations $\delta = \h^{1+\mu-2\rho}$, $\h^{\rho}\frac{\prtl}{\prtl x_j} = \frac{\prtl}{\prtl y_j} $, we obtain that if $\ell-m = j$,
\begin{equation*}
\delta^{j}
\Big\langle \mbox{$\left(\frac{\prtl^2\phi}{\prtl y^2}(0)\right)^{-1}$} D_y,D_y \Big\rangle^\ell [G(y)]^{m} u(y) \Big|_{y=0} = \h^{j} \Big\langle\mbox{$\left(\frac{\prtl^2\phi_\h}{\prtl x^2}(0)\right)^{-1}$} D_x, D_x\Big\rangle^{\ell}  [g(x)]^m u(x) \Big|_{x=0}.
\end{equation*}
In other words, $\delta^j \Msc_{j} u|_{y=0} = \h^j \Lsc_{j,\h} u|_{x=0}$, where $\Lsc_{j,\h}$ is defined in \eqref{Ljdef}.

We lastly observe that again by the relation $\delta = \h^{1+\mu-2\rho}$ and  \eqref{hessianchange}
\begin{equation*}
\h^{\rho d}\frac{(2\pi \delta)^{\frac d2}}{{\det\!^{1/2}\big(\frac 1i \mbox{$\frac{\prtl^2\phi}{\prtl y^2}(0)$}  \big) }} = \frac{(2\pi \h)^{\frac d2}}{{\det\!^{1/2}\big(\frac 1i \mbox{$\frac{\prtl^2\phi}{\prtl x^2}(0)$}  \big) }}
\end{equation*}
Consequently, the desired estimate \eqref{spreduction} now follows by multiplying both sides of \eqref{yexpn} by $\h^{\rho d}$.
\end{proof}

\section{The Van Vleck Formula on Ehrenfest Time Scales}\label{S:VV}
\subsection{Regularity of the Hamiltonian flow}\label{SS:phasereg}
Recall our assumption \ref{FlowHyp} in \S\ref{SS:VVIntro} that $\kappa_t$ satisfies 
\begin{equation}\label{flowderivs}
	|\prtl_z^\gamma \kappa_t(z)| = |(\prtl_z^\gamma q_t(z),\prtl_z^\gamma p_t(z)) | \lesssim_{\gamma} \omega(t)^{|\gamma|}\text{ for }  |\gamma| \geq 1
\end{equation}
In this subsection, we examine the implications of this hypothesis for the regularity of the Herman-Kluk phase $\Phi$ introduced in \eqref{phaseintro}
\begin{equation}\label{HKphasedef}
	\Phi(t,x,y,q,p) := S(t,q,p) + p_t(q,p)\cdot(x-q_t(q,p)) - p\cdot(y-q) + \frac i2 \big(|x-q_t(q,p)|^2 + |y-q|^2 \big).
\end{equation}
In particular, we will show how the phase satisfies the hypotheses of Theorem \ref{T:spthm}.

Since $\kappa_t$ preserves the canonical symplectic 2-form on $T^*\RR^d$, its differential is given by the symplectic matrix
\begin{equation}\label{kappajacobian}
	d_z \kappa_t(z)
	=
	\begin{bmatrix}
		\frac{\prtl q_t}{\prtl q}(z) & \frac{\prtl q_t}{\prtl p}(z) \\
		\frac{\prtl p_t}{\prtl q}(z) & \frac{\prtl p_t}{\prtl p}(z)
	\end{bmatrix}	
	=:
	\begin{bmatrix}
		A_t(z) & B_t(z)\\
		C_t(z) & D_t(z)
	\end{bmatrix},
\end{equation}
so the matrix function $A_t(z)$ denotes $\frac{\prtl q_t}{\prtl q}(z)$, $B_t(z)$ denotes $ \frac{\prtl q_t}{\prtl p}(z)$, etc.  Recall that symplectic matrices in block form satisfy the fundamental identities 
\begin{equation}\label{symplecticID}
	A^T C = C^T A, \;B^T D = D^TB, \text{ and } A^TD-C^TB = I.
\end{equation}
 In what follows, we make use of the following bounds which follow from \eqref{flowderivs}:
\begin{equation}\label{ABCDbds}
	\|\prtl_z^\alpha A_t(z)\| + \|\prtl_z^\alpha B_t(z)\| + \|\prtl_z^\alpha C_t(z)\| + \|\prtl_z^\alpha D_t(z)\| \lesssim_\alpha \omega(t)^{|\alpha|+1} \leq  \h^{-\lambda(|\alpha| +1)} , \quad |\alpha| \geq 0.
\end{equation}

\begin{lemma}\label{L:phasereg}
	Suppose Assumption \ref{FlowHyp} in \S\ref{SS:VVIntro} is satisfied and $t$ is such that $\omega(t) \leq \h^{-\lambda}$. Then for $x,z$ satisfying $|x-q_t(z)| <1$, we have that   
	\begin{equation*}
		|\prtl_z^\alpha \Phi(t,x,y,z)| \lesssim_\alpha \omega(t)^{|\alpha|} \leq \h^{-\lambda|\alpha|}, \quad |\alpha| \geq 2.
	\end{equation*}
\end{lemma}

\begin{proof}
	Differentiating the action $S$ in \eqref{actionintro} with respect to $q,p$ gives
	\begin{equation*}
		\frac{\prtl S}{\prtl q} = A_t^T p_t - p, \qquad \frac{\prtl S}{\prtl p} = B_t^T p_t.
	\end{equation*}
	It is then seen that the first derivatives of $\Phi$ simplify to 
	\begin{equation}\label{Phifirstderivs}
		\begin{split}
					\frac{\prtl \Phi}{\prtl q} &= (iA_t-C_t)^T(q_t-x) + i(q-y),\\
			\frac{\prtl \Phi}{\prtl p} &= (iB_t-D_t)^T(q_t-x) + (q-y).
		\end{split}
	\end{equation}
	
	If $|x-q_t(z)| <1$, then the Hessian of $\Phi$ satisfies
	\begin{equation}\label{HessianHK}
		\frac 1i \frac{\prtl^2\Phi}{\prtl z^2}
		=
		\begin{bmatrix}
			(A_t+iC_t)^TA_t + I & (A_t+iC_t)^TB_t \\
			(B_t+iD_t)^TA_t-iI & (B_t+iD_t)^TB_t
		\end{bmatrix}
		+ R(t,z),
	\end{equation}
	where $R$ vanishes to first order at $x=q_t(z)$ and the $|\gamma| = |\alpha|+1$ case of \eqref{ABCDbds} gives
	\begin{equation*}
		|\prtl^\gamma R| \lesssim_\gamma \omega(t)^{|\gamma|+2} \leq  \h^{-\lambda(|\gamma|+2)}, \quad |\gamma| \geq 0.
	\end{equation*}
	The proof is now concluded by the following bound, which follows by \eqref{ABCDbds} and the Leibniz formula
	\begin{equation*}
		\Big|\prtl_z^\gamma\Big(\frac 1i \frac{\prtl^2\Phi}{\prtl z^2}-R\Big)\Big| \lesssim_\gamma \omega(t)^{|\gamma|+2} \leq \h^{-\lambda(|\gamma|+2)}, \quad |\gamma| \geq 0.
	\end{equation*}
\end{proof}
	
\begin{lemma}\label{L:invHessPhibd}
	Suppose Assumption \ref{FlowHyp} in \S\ref{SS:VVIntro} is satisfied and $t$ is such that $\omega(t) \leq \h^{-\lambda}$.   Suppose further that Assumption \ref{PtsHyp} is satisfied and that $\eta \in \widetilde{\Xi}_t(x,y)$.  Then the inverse of the Hessian of $\Phi$ at $(q,p) = (y,\eta)$ satisfies
	\begin{equation}\label{invHessPhibd}
		\Big\| \Big(\frac{\prtl^2 \Phi}{\prtl z^2}\Big|_{(q,p) = (y,\eta)} \Big)^{-1} \Big\| \lesssim \h^{-\delta-2\lambda}
	\end{equation}
\end{lemma}
\begin{proof} 
	We begin by defining the matrix functions $Y_t(z),Z_t(z)$ by
	\begin{equation}\label{YZdefs}	
		Y_t := A_t+D_t +i(B_t-C_t), \quad Z_t := A_t-D_t +i(B_t+C_t) .
	\end{equation}
	As observed in \cite[Ch. 3, Lemma 22]{CombescureRobertBook}, the identities \eqref{symplecticID} then imply that $Y_t^*Y_t = Z_t^* Z_t + 4I$, where the $*$ denotes the adjoint matrix.  In particular this implies that $Y_t$ is not only invertible, but for any vector $v \in \mathbb{C}^d$, $|Y_t v| \geq 2|v|$. Hence all singular values of $Y_t$ are bounded below by 2, which in turn implies that 
	\begin{equation}\label{Yinvbd}
		\| Y_t^{-1}\| \leq 2.
	\end{equation}
	Moreover, $Y_t,Z_t$ satisfy the bounds in \eqref{ABCDbds}, 
	\begin{equation}\label{YZbd}
			\| \prtl_z^\alpha Y_t (z) \| +	\| \prtl_z^\alpha Z_t (z) \| \lesssim_{\alpha} \omega(t)^{|\alpha|+1} \leq \h^{-\lambda(|\alpha|+1)}, \quad |\alpha| \geq 0.
	\end{equation}

	We now compute the desired inverse Hessian using \eqref{HessianHK}.  Since $\eta \in \widetilde{\Xi}_t(x,y)$, the remainder $R$ there vanishes at $(q,p) = (y,\eta)$.  It can then be verified by direct computation that 
	\begin{equation*}
	\Big(\frac 1i \frac{\prtl^2 \Phi}{\prtl z^2}\Big|_{(q,p) = (y,\eta)} \Big)^{-1} = 
		\begin{bmatrix}
			\frac 12(I-Z_t^TY_t^{-*}) & \frac i2(I+Z_t^TY_t^{-*})  \\ 
			B_t^{-1}(Y_t^{-*}-\frac 12 A_t(I-Z_t^TY_t^{-*}))& -iB_t^{-1}(Y_t^{-*}+\frac 12 A_t(I-Z_t^TY_t^{-*}))
		\end{bmatrix}.
	\end{equation*}
	Here $Y_t^{-*}$ denotes the inverse of the adjoint $Y_t^*$.  The desired bound \eqref{invHessPhibd} then follows from \eqref{Yinvbd} and \eqref{YZbd}, with the largest possible contributions coming from the expressions $B_t^{-1}A_t Z_t^T Y_t^{-*}$.
\end{proof}

\begin{remark}
	Lemmas \ref{L:phasereg} and \ref{L:invHessPhibd} suggest that we should apply Theorem \ref{T:spthm} with 
	\begin{equation}\label{parameterdefs}
		\begin{split}
		&\mu = 2\lambda, \nu = \lambda, \text{ and } \sigma = \delta + 2\lambda,\text{ so that }  \\
		&\rho = \delta + 5\lambda \text{ and } 5\mu + 6\sigma + 2\nu = 24\lambda +6\delta.
		\end{split}
	\end{equation}	 
	We will make use of this convention in what follows.  This means the hypothesis $24\lambda + 6 \delta <1$ in Theorem \ref{T:VVthm} implies the hypothesis \eqref{asymptotichyp}.
\end{remark}

We conclude this subsection with some regularity estimates (\eqref{Gammabds} below) needed for an integration by parts argument in \S\ref{SS:HermanKlukApp} using \eqref{transposePhi}. 
Begin with the transpose of \eqref{kappajacobian}
\begin{equation*}
	\begin{bmatrix}
		A_t^T & C_t^T\\
		B_t^T & D_t^T
	\end{bmatrix}.
\end{equation*}
Since this transposed matrix is also symplectic, the matrices 
\begin{equation}\label{siegellemma}
	(A_t+iC_t)^T \text{ and } (B_t+iD_t)^T \text{ are both invertible},
\end{equation}
see \cite[Ch.3, Lemma 21]{CombescureRobertBook}. From this, we define matrices $\tilde Y_t$, $\tilde Z_t$ analogously to \eqref{YZdefs}
\begin{equation*}
	\tilde Y_t := A_t^T+D_t^T +i(C_t^T-B_t^T), \quad \tilde Z_t := A_t^T-D_t^T +i(C_t^T+B_t^T) .
\end{equation*}
As before, we have $\tilde Y_t^* \tilde Y_t = \tilde Z_t^* \tilde Z_t +4I$ so that $ \|\tilde Y_t^{-1} \| \leq 2$ as in \eqref{Yinvbd}.  We also observe that $\tilde Y_t$, $\tilde Z_t$ satisfy the same bounds as in \eqref{YZbd}
\begin{equation*}
		\|\prtl_z^\alpha \tilde Y_t(z)\| + \|\prtl_z^\alpha \tilde Z_t(z)\| \lesssim_\alpha \omega(t)^{|\alpha|+1}  \leq \h^{-\lambda(|\alpha| +1)} .
\end{equation*}
Consequently an induction on  $|\alpha|$, using the identity $0=\prtl_z^{\alpha} (\tilde Y_t \tilde Y_t^{-1})$ for $|\alpha|>0$, gives that
\begin{equation*}
	\|\prtl_z^\alpha (\tilde Y_t^{-1}(z))\|  \lesssim_\alpha \omega(t)^{2|\alpha|} \leq  \h^{-2\lambda|\alpha|} .
\end{equation*}
The Liebniz rule then gives that if $\tilde W_t := \tilde Z_t \tilde Y_t^{-1}$, then
\begin{equation}\label{Wtildebds}
	\|\prtl_z^\alpha \tilde W_t(z)\|  \lesssim_\alpha \omega(t)^{2|\alpha|+1} \leq \h^{-\lambda(2|\alpha|+1)} .
\end{equation}

We now define $\Gamma_t = (B_t+iD_t)^T(A_t+iC_t)^{-T}$, which is well known to lie in the Siegel space of symmetric matrices with positive definite imaginary part (see e.g. \cite[Ch. 3, Lemma 21]{CombescureRobertBook}).  Given \eqref{Phifirstderivs}, we have
\begin{equation}\label{transposePhi}
	(\Gamma_t+iI)^{-1} \Big(\Gamma_t \frac{\prtl \Phi}{\prtl q} - \frac{\prtl \Phi}{\prtl p}\Big) = i(q-y),
\end{equation}
which will be used in an integration by parts argument in \S\ref{SS:HermanKlukApp}.  Indeed, a routine algebraic computation (cf. \cite[Ch. 3, Lemma 23]{CombescureRobertBook}) reveals that $\Gamma_t +iI$ is invertible with
\begin{align*}
	(\Gamma_t +iI)^{-1} &= \frac 1i (A_t+iC_t)^T \tilde Y_t^{-1}= \frac{1}{2i} (I+\tilde W_t),\\
	\Gamma_t(\Gamma_t +iI)^{-1} &= \frac 1i (B_t+iD_t)^T \tilde Y_t^{-1}= \frac{1}{2} (I-\tilde W_t).
\end{align*}
Consequently, \eqref{Wtildebds} gives
\begin{equation}\label{Gammabds}
	\|\prtl_z^\alpha ( (\Gamma_t +iI)^{-1})\| + \|\prtl_z^\alpha ( \Gamma_t (\Gamma_t +iI)^{-1})\| \lesssim_\alpha \omega(t)^{2|\alpha|+1}.
\end{equation}

\subsection{The Herman-Kluk Approximation}\label{SS:HermanKlukApp}
In this section, we review the results of Robert \cite[Theorem 1.4]{RobertHK} which justify the Herman-Kluk approximation for $\U_t$, then examine the effect of composing this with $\widehat{\Theta}$.  The following is a modest restatement of his theorem.
\begin{theorem}[Robert]
	Suppose Assumptions \ref{HamHyp} and \ref{FlowHyp} from \S\ref{SS:VVIntro} are satisfied and that $t$ is such that $\omega(t) \leq \h^{-\lambda}$ with $24\lambda < 1$.  With the notation as in \S\ref{SS:phasereg}, let $\bar{Y}_t$ denote the complex conjugate of the matrix $Y_t$ in \eqref{YZdefs}.
	There exists a sequence of symbols $\{a_j(t,z)\}_{j=0}^\infty$ 
	so that if $\U_{N,t}$ is the operator with integral kernel given by
	\begin{equation}\label{propkernel}
		K_{N,t} (x,y) = \frac{1}{(2\pi\h)^{3d/2}}\int_{T^*\RR^d} e^{\frac i\h \Phi(t,x,y,z)} \det\!^{1/2} \bar{Y}_t(z)\left( \sum_{j=0}^N a_j(t,z)\h^j\right) \,dz
	\end{equation}
	then there exists $\veps_0>0$ such that 
	\begin{equation}
		\|\U_t-\U_{N,t} \|_{L^2(\RR^d) \to L^2(\RR^d)} \lesssim_{N} \h^{\veps_0(N+1)} .
	\end{equation}
	Here $a_0(t,z) = \exp(-i\int_0^tH_1(s,z_s(z))\,ds)$ and the symbols $a_j$ satisfy
	\begin{equation}\label{ajbdslongtime}
		|\prtl_z^\gamma a_j(t,z)| \lesssim_{j,\gamma} \omega(t)^{4j+|\gamma|}.
	\end{equation}
\end{theorem}

\begin{remark}\label{R:HKconv}
	We remark on the harmless differences between the statement here and that in \cite{RobertHK}.  The analysis in \cite{RobertHK} is rooted in the matrix $M_t(z) := -i\bar{Y}_t(z)$, showing that the sequence of $\tilde a_j :=(\det\!^{1/2} \bar{Y}_t )  a_j $ satisfies (cf. (3.24) in that work\footnote{Here we display the effect of including all terms in the asymptotic expansion $H \sim \sum_{j=0}^{\infty} \h^j H_j$, whereas this is treated implicitly in \cite[\S3]{RobertHK}.})
	\begin{equation*}
		\begin{split}
			\prtl_t \tilde a_{0}(t,z) &= \frac 12 \text{tr}\,\big((\prtl_t M_t) M_t^{-1}\big)\tilde a_{0}(t,z) -iH_1(\kappa_t(z))\tilde a_0(t,z)\\
			\prtl_t \tilde a_{j+1}(t,z) &= \frac 12 \text{tr}\,\big((\prtl_t M_t) M_t^{-1}\big)\tilde a_{j+1}(t,z) -iH_{j+1}(\kappa_t(z))\tilde a_{j+1}(t,z) + b_j(t,z), \quad j \geq 0,
		\end{split}
	\end{equation*}
	where $a_j(0,z) \equiv 0$ when $j \geq 1$ and $b_j$ depends on quantities related to $\kappa_t$, $H_0,\dots,H_j$, and $\tilde a_0,\dots,\tilde a_j$.  However, since $\text{tr}\,((\prtl_t M_t) M_t^{-1}) = \text{tr}\,((\prtl_t \bar{Y}_t)\bar{Y}_t^{-1})$, we arrive at equations for $a_j$ through the ansatz $\tilde a_j=(\det\!^{1/2} \bar{Y}_t )a_j$.  Since $\det\!^{1/2} \bar{Y}_0(z) \equiv 2^{\frac d2}$, we have the initial condition $a_0(0,z) = 1$ in light of the identity $2^{-d}(\pi\h)^{-\frac{3d}{2}}\int e^{\frac i\h \Phi(0,x,y,z)} \,dz = \delta(x-y)$.  In particular, we have 
	\begin{equation}\label{a0expression}
		a_0(t,z) = \exp(-i\int_0^t H_1(z_s(z))ds).
	\end{equation}
	We also note that in \cite{RobertHK}, the determinant factors are absorbed into the amplitude, but here we have factored them out, which  makes it easier to examine the effect of applying stationary phase. 
\end{remark}

Given Theorem \ref{HKphasedef}, the Borel lemma furnishes an amplitude\footnote{Here and below, the roles of $\tilde a$ and $\tilde a_j$ will be much different than their role in Remark \ref{R:HKconv}.} $\tilde a(t,z) \sim \sum_{j=0}^\infty \h^ja_j(t,z)$ with $| \prtl_z^\alpha \tilde a(t,z)| \lesssim_\alpha \omega(t)^{4|\alpha|}$  such that if $\tilde\U_t$ is the operator with kernel  
\begin{equation}\label{Ktildeinitialdef}
	\tilde K_t(x,y) = \frac{1}{(2\pi\h)^{3d/2}}\int_{T^*\RR^d} e^{\frac i\h \Phi(t,x,y,z)} \det\!^{1/2} \bar{Y}_t(z)\tilde a(t,z) \,dz
\end{equation}
then $\|\U_t - \tilde{\U}_t\|_{L^2 \to L^2} = O(\h^\infty)$.
Recall that \eqref{Yinvbd}, \eqref{YZbd} imply that 
\begin{equation}\label{Ybarbds}
	\|\bar{Y}_t^{-1}\| \leq 2 \text{ and } \|\prtl^\alpha \bar{Y}_t\| \lesssim_\alpha \h^{-\lambda(|\alpha|+1)} . 
\end{equation}	
Thus by the Jacobi formula (on differentiating determinants), the total amplitude satisfies
\begin{equation*}
	\big| \det\!^{-1/2} \bar{Y}_t(z)\prtl_z^\alpha\big(\det\!^{1/2} \bar{Y}_t(z)\tilde a(t,z)\big)\big| \lesssim_\alpha \omega(t)^{4|\alpha|} \leq \h^{-4\lambda|\alpha|}.
\end{equation*}

We now want to observe the effect of composing $\tilde\U_t$ with $\widehat{\Theta}$.  In preparation, we begin with a lemma which is a variation on \cite[Lemma 3.3]{RobertHK}:
\begin{lemma}\label{IBPLemma}
	Let $\alpha \neq 0$ be any multi-index.  Suppose $b$ is a $C^{|\alpha|}$, integrable function and that  each derivative of order up to $|\alpha|$ is also integrable.  Then
	\begin{equation*}
		\int e^{\frac i\h \Phi(t,x,y,z)} (q-y)^\alpha b(z)\,dz = 
		\sum_{\frac{|\alpha|}{2} \leq j \leq |\alpha|} \h^j \sum_{|\beta| \leq 2j-|\alpha|} 
		\int e^{\frac i\h \Phi(t,x,y,z)} f_{\alpha,j,\beta} (t,z) \prtl_z^\beta b(z) \,dz
	\end{equation*}
	where $j \in \mathbb{N}$ and $f_{\alpha,j,\beta} $ is a collection of functions depending on $\kappa_t$ satisfying
	\begin{equation*}
		|\prtl_z^\gamma f_{\alpha,j,\beta} (t,z)| \lesssim_{\alpha,j,\beta} \omega(t)^{j+2(2j-|\alpha|-|\beta|)+2|\gamma|}.
	\end{equation*}
\end{lemma}
\begin{proof}
	The proof follows by induction on $|\alpha|$, using \eqref{transposePhi} to integrate by parts, along with the bounds in  \eqref{Gammabds}.
\end{proof}

The next proposition furnishes the operator and kernel $\V_t,L_t$ desired in Theorem \ref{T:VVthm}.
\begin{proposition}\label{P:composition}
There exists an amplitude $a(t,z)$ satisfying
\begin{equation}\label{finalamplitude}
			\big| \det\!^{-1/2} \bar{Y}_t(z)\prtl_z^\alpha\big(\det\!^{1/2} \bar{Y}_t(z) a(t,z)\big)\big| \lesssim_\alpha\omega(t)^{4|\alpha|} \leq \h^{-4\lambda|\alpha|},\quad			\supp(a(t,\cdot)) \subset \supp(\Theta),
\end{equation}
such that the kernel $L_t(x,y)$ given by 
\begin{equation}\label{Ldef}
	L_t(x,y) = \frac{1}{(2\pi\h)^{3d/2}}\int_{T^*\RR^d} e^{\frac i\h \Phi(t,x,y,z)} \det\!^{1/2} \bar{Y}_t(z) a(t,z) \,dz
\end{equation}
defines an operator $\V_t$ such that 
\begin{equation}\label{lasterror}
	\|\V_t - \tilde\U_t \circ \widehat{\Theta} \|_{L^2 \to L^2} = O(\h^\infty).
\end{equation}
\end{proposition}
\begin{proof}
The kernel of $\widehat{\Theta}$ is given by 
\begin{equation}\label{Thetakernel}
	\frac{1}{(2\pi\h)^d} \int e^{\frac i\h (v-y)\cdot \xi} \Theta(\xi)\,d\xi.
\end{equation}
Hence the kernel of the composition $\tilde \U_t \circ \widehat{\Theta}$ can be written 
\begin{multline}\label{kernelcomp}
	\frac{1}{(2\pi\h)^{5d/2}}\int_{\RR^{2d}} \int_{\RR^{d}} \int_{\RR^{d}} e^{\frac i\h (\Phi(t,x,v,z)  +(v-y)\cdot \xi)} 
	\det\!^{1/2} \bar{Y}_t(z)\tilde{a}(t,z) \Theta(\xi)\,dzd\xi dv\\
	=\frac{1}{(2\pi\h)^{3d/2}}\int_{\RR^{2d}}  e^{\frac i\h (S(t,z) + p_t\cdot(x-q_t) +\frac i2|x-q_t|^2)} 
	\det\!^{1/2} \bar{Y}_t(z)\tilde{a}(t,z) J(z,y)dz,
\end{multline}
where 
\begin{equation*}
	J(z,y) := 	\frac{1}{(2\pi\h)^{d}} \iint e^{\frac i\h (p\cdot(q-v) +\frac i2|v-q|^2 +(v-y)\cdot \xi)}\Theta (\xi)\,d\xi dv.
\end{equation*}

To further analyze $J$, we translate variables $(v,\xi) \mapsto (v+q,\xi+p)$, so that Gaussian integration in $v$ gives:
\begin{equation*}
			J(z,y) = 	\frac{e^{\frac i\h p\cdot(q-y)} }{(2\pi\h)^{d}} \iint e^{\frac i\h ( (v +q -y)\cdot \xi +\frac i2|v|^2 )}\Theta(\xi+p)\,dv d\xi 
				 = 	\frac{e^{\frac i\h p\cdot(q-y)}}{(2\pi\h)^{d/2}} \int e^{\frac i\h (q-y)\cdot \xi -\frac{1}{2\h}|\xi|^2}\Theta(\xi+p) \,d\xi .
\end{equation*}
Next, we take a Taylor expansion of $\Theta$
\begin{equation*}
	\Theta(\xi+p) = \sum_{|\alpha| <N} \frac{\xi^\alpha}{\alpha!} \prtl^\alpha \Theta(p) + \sum_{|\alpha| = N} \frac{\xi^\alpha}{\alpha!} \int_0^1 N(1-s)^{N-1}\prtl^\alpha \Theta(s\xi+p)\,ds
\end{equation*}
By induction, $D_{\tilde v}^\alpha e^{-|\tilde v|^2/2} = P_\alpha(\tilde v)e^{-|\tilde v|^2/2}$ for some polynomial $P_\alpha$ of the same parity and degree as $|\alpha|$ (and $P_0(\tilde v) \equiv 1$).  Applying this with $\h^{1/2}\tilde v = v-q$, the chain rule implies that the contribution of the sum over $|\alpha|<N$ to $J(z,y)$ is thus
\begin{equation}\label{taylorleading}
	e^{\frac i\h p\cdot(q-y)}\sum_{|\alpha| <N} \frac{\prtl^\alpha \Theta(p) }{\alpha!} (\h D_v)^\alpha e^{-\frac{|v-q|^2}{2\h}}\Big|_{v=y} = e^{\frac i\h p\cdot(q-y)-\frac{|y-q|^2}{2\h}} \sum_{|\alpha| <N} \frac{\prtl^\alpha \Theta(p) }{\alpha!} \h^{\frac{|\alpha|}{2}} P_\alpha \Big(\frac{y-q}{\sqrt{\h}}\Big) 
\end{equation}
We also observe that for any $s \in (0,1)$, 
\begin{equation}\label{taylorerror}
	\Big| \frac{1}{(2\pi \h)^{d/2}} \int e^{\frac i\h (q-y)\cdot \xi -\frac{1}{2\h}|\xi|^2} \xi^\alpha \prtl^\alpha \Theta(s\xi+p) \,d\xi \Big| \lesssim_{N'} \h^{\frac{|\alpha|}{2}}(1+\h^{-\frac 12}|q-y| + |p|)^{-N'}
\end{equation}
where the decay in $|p|$ results from the compact support of $\Theta$.   Hence the contribution of the Taylor remainder to \eqref{kernelcomp} results in a kernel $R_{N,t}(x,y)$ satisfying
\begin{equation}\label{taylorerrorcontrib}
	|R_{N,t}(x,y)| \lesssim_{N'} \h^{\frac{N}{2}-c_d}\int_{\RR^{2d}} \big(1+\h^{-\frac 12}(|x-q_t(q,p)| + |y-q|) +|p| \big)^{-N'}\,dq dp,
\end{equation}
for some $c_d$ depending only on $d$. Hence by Young's inequality, the larger we take $N$ in the Taylor expansion the smaller the error generated by $R_{N,t}$ as an operator on $L^2(\RR^d)$.

The contribution of \eqref{taylorleading} to the kernel of $\tilde\U_t \circ \widehat{\Theta}$ is thus
\begin{equation*}
 \frac{1}{(2\pi\h)^{3d/2}}	\sum_{|\alpha| <N} \frac{\h^{|\alpha|/2}}{\alpha! }\int_{\RR^{2d}}  e^{\frac i\h \Phi(t,x,y,z)} \det\!^{1/2} \bar{Y}_t(z)\tilde{a}(t,z) \prtl^\alpha \Theta(p) P_\alpha \Big(\frac{y-q}{\sqrt{\h}}\Big) dz.
\end{equation*}
Repeated use of Lemma \ref{IBPLemma} and \eqref{taylorerrorcontrib} then gives a sequence of amplitudes $\{\tilde a_j(t,z)\}_{j=0}^\infty$ so that if $L_{N,t}(x,y)$ is the kernel
\begin{equation*}
	L_{N,t}(x,y) := \frac{1}{(2\pi\h)^{3d/2}}	 \int e^{\frac i\h \Phi(t,x,y,z)} \det\!^{1/2} \bar{Y}_t(z)\Big(\sum_{j=0}^{N} \h^j\tilde{a}_j(t,z) \Big)dz,
\end{equation*}
then the corresponding operator $\V_{N,t}$ satisfies $\| \V_{N,t} -\tilde \U_t \circ \widehat{\Theta} \|_{L^2 \to L^2} \lesssim_N \h^{\veps_1 (N+1)}$ for some $\veps_1>0$.  One more application of the Borel lemma now furnishes an amplitude $a(t,z)$ so that \eqref{Ldef} satisfies the desired properties \eqref{finalamplitude} and \eqref{lasterror}.

We conclude by observing that since $P_0 \equiv 1$, we have $a(t,z) - \Theta(p)a_0(t,z) = O(\h^{1-5\lambda})$  where  $a_0$ is in \eqref{a0expression}.  Given this, the estimate \eqref{vvleading} above can be seen to follow from the arguments below since $\Lsc_{0,\h} \equiv 1$ in Theorem \ref{T:spthm}.
\end{proof}

\subsection{Non-stationary points in the Herman-Kluk approximation}
In this section, we begin our analysis of the oscillatory integral in \eqref{Ldef}, showing that the contribution of the integral away from points $\eta \in \widetilde{\Xi}(x,y)$ is $O(\h^\infty)$.  This reduces Theorem \ref{T:VVthm} to an application of the stationary phase asymptotics in Theorem \ref{T:spthm}. We begin with quantitative version of the inverse function theorem, the only place in this work where we allow the variables $q,p$ to play different roles.
\begin{lemma}\label{L:InvFT}
	Suppose $F: \RR^{d} \supset B(p_0,r) \to \RR^{d}$ is a $C^2$ map and there exists $0<R_1,R_2<\infty$ such that $\|(\frac{\prtl F}{\prtl p}(p_0))^{-1}\|\leq R_1^{-1}$ and that $\sup_{p \in B(p_0,r),|\alpha|=2}|\prtl^\alpha F(p)| \leq R_2^{-1}$.  Then there exists $c_\d>0$ depending only on $\d$ such that $F$ is injective on any ball of radius $r_0$ satisfying 
	\begin{equation*}
		0<r_0 < \min(c_\d R_1R_2,r).
	\end{equation*}  Moreover, the image $F(B(p_0,r_0))$ contains a ball about $F(p_0)$ of radius $\frac 12 r_0 R_1$.
\end{lemma}
\begin{proof}
	The lemma follows from the contraction mapping approach to the inverse function theorem in \cite[Theorem 9.24]{RudinPrinciples}; we merely outline the main steps in the proof here.  For $q \in \RR^d$, define the mapping $\varphi_q(p) :=  p + (\frac{\prtl F}{\prtl p}(p_0))^{-1}(q-F(p))$. The solutions of $F(p)=q$ are identical to the fixed points of $\varphi_q$.  Moreover, if $0<r_0 < r$ and
	\begin{equation}\label{contractive}
		\Big\|\Big(\frac{\prtl F}{\prtl p}(p_0)\Big)^{-1}\Big\|
		\Big(\sup_{p \in B(p_0,r_0)}\Big\|\frac{\prtl F}{\prtl p}(p) - \frac{\prtl F}{\prtl p}(p_0)\Big\|\Big)
		\leq \frac 12,
	\end{equation}
	then $\varphi_q$ is a contraction on $\overline{B(p_0,r_0)}$, meaning any fixed points $p \in \overline{B(p_0,r_0)}$ must be unique.  It follows that $F$ is injective on $\overline{B(p_0,r_0)}$.  Moreover, \eqref{contractive} gives that the image $F(\overline{B(p_0,r_0)})$ contains a ball of radius $r_0(2 \|(\frac{\prtl F}{\prtl p}(p_0))^{-1}\| )^{-1}$ about $F(p_0)$, which is larger than the concentric ball of  radius $\frac 12 r_0 R_1$.  Indeed, if $q$ lies in this ball, it is seen that $\varphi_q$ maps the closed ball $\overline{B(p_0,r_0)}$ to itself, at which point the claim follows from the contraction mapping fixed point theorem.
		
	Since $F$ is $C^2$, there exists a constant $C_\d$ depending only $\d$ such that \eqref{contractive} follows if $r_0$ satisfies
	\begin{equation*}
		C_\d \,r_0\, \Big\|\Big(\frac{\prtl F}{\prtl p}(p_0)\Big)^{-1}\Big\|
		\Big(\sup_{p \in B(p_0,r_0),|\alpha|=2}|\prtl^\alpha F(p)|\Big)\leq \frac 12,
	\end{equation*}
	which in turn follows from $C_d r_0 R_1^{-1}  R_2^{-1} \leq \frac 12$.  Setting $c_d = (2C_d)^{-1}$ thus establishes the result.
\end{proof}

Recall from \eqref{parameterdefs} that we have set $\mu = 2\lambda$, $\nu = \lambda$, $\sigma = \delta + 2\lambda$, $\rho = \delta + 5\lambda$. 
We will observe two consequences of Lemma \ref{L:InvFT} with $F(p) = q_t(y,p)$, $R_1 := c\h^{\sigma}$, $R_2 := c\h^{2\lambda}$, and $1/c$ sufficiently large.  The first observation is that the preimage
\begin{equation}\label{preimage}
	\tilde \Xi_t(x,y) = \{p \in \supp(\Theta)  + B(0,\h^\rho): x= q_t(y,p)\}
\end{equation}
is a finite, discrete set, with cardinality  $O(\h^{-(\sigma+2\lambda)\d})$.  Moreover, if $\eta,\tilde \eta \in \tilde \Xi_t(x,y) $ are distinct elements, we have $|\eta -\tilde \eta| \gtrsim h^{\sigma+2\lambda } \geq \h^\rho$.

Next. since $\rho > \delta+4\lambda$, the second part of Lemma \ref{L:InvFT} then shows that if $p\in  \supp(\Theta)$ is such that $|p-\eta| \geq \h^{\rho}$ for all $\eta \in \tilde \Xi_t(x,y)$, then
for $\h$ small enough,
\begin{equation}\label{pdistance}
 	\frac 12\h^{\rho + \delta +2\lambda } \leq |x-q_t(y,p)|.
\end{equation}
Indeed, if we had $ |x-q_t(y,p)| < \frac 12 \h^{\rho + \delta +2\lambda }$, then Lemma \ref{L:InvFT} would furnish $\eta \in B(p,\h^\rho)$ such that $x = q_t(y,\eta)$, contradicting that $p$ is at a distance of at least $\h^\rho$ to the set in \eqref{preimage}.  Moreover, \eqref{pdistance} implies that 
\begin{equation}\label{pdistance2}
\frac 14 \h^{2(\rho + \delta +2\lambda )} \leq |x-q_t(q,p)|^2 + |q_t(y,p)-q_t(q,p)|^2 \lesssim |x-q_t(q,p)|^2 + \h^{-2\lambda}|y-q|^2
\end{equation}

Let $\chi$ be a bump function identically one on the ball of radius $1/2$ about 0 in $\RR^d$ and supported in the ball of radius 1 about 0. For each $\eta \in \tilde \Xi_t(x,y)$, we now define $\psi_\eta(q,p) = \chi(\h^{\rho}(p-\eta))\chi(\h^\rho(y-q))$ and $\psi_0(q,p) = (1- \sum_{\eta}\psi_\eta(q,p))$.  We now claim that if $(q,p) \in \supp(\psi_0)$, then
\begin{equation}\label{psi0supp}
 	\h^{2(\rho + \delta +3\lambda )} \lesssim |x-q_t(q,p)|^2 + |y-q|^2.
\end{equation}
Indeed, \eqref{pdistance2} shows there exists $\veps_1>0$ sufficiently small such that if $|y-q| < \veps_1\h^{\rho+\delta+3\lambda}$, then $|x-q_t(q,p)|^2 \gtrsim \h^{2(\rho + \delta +3\lambda )}$.  If $|y-q|\geq \veps_1\h^{\rho+\delta+3\lambda}$ instead, then the desired inequality is trivial.

The hypothesis $6\delta + 24\lambda <1$ in Theorem \ref{T:VVthm} means that
$
	2(\rho+\delta+3\lambda) = 4\delta+16\lambda  <1.
$
Since the bound \eqref{psi0supp} implies that $\Im\Phi \geq \h^{2(\rho + \delta +3\lambda )}$, we have
\begin{equation*}
	\frac{1}{(2\pi\h)^{3d/2}}\int_{T^*\RR^d} e^{\frac i\h \Phi(t,x,y,z)} \det\!^{1/2} \bar{Y}_t(z)a(t,z) \psi_0(z) \,dz = O(\h^\infty).
\end{equation*}

Theorem \ref{T:VVthm} now follows from the following result, which we prove in \S\ref{SS:SPandVVcon} below.
\begin{theorem}\label{T:locasymptotic} Given $\eta \in \widetilde{\Xi}_t (x,y)$, set
	\begin{equation*}
		L_{\eta,t} (x,y) := \int_{T^*\RR^d} e^{\frac i\h \Phi(t,x,y,z)} \det\!^{1/2} \bar{Y}_t(z)a(t,z) \psi_\eta(z) \,dz.
	\end{equation*}
There exists $\theta_\eta \in \mathbb{Z}$ and a sequence $\{b_k(y,\eta)\}_{k=1}^\infty$ such that for any integer $N \geq 1$
	\begin{equation}\label{spetaloc}
		\left| L_{\eta,t} (x,y) -
		(2\pi\h)^{d} e^{\frac i\h S_\eta-i\frac{\pi}{2} \theta_\eta }  \Big|\det \frac{\prtl^2 S_\eta}{\prtl x \prtl y}\Big|^{1/2}\sum_{k=0}^{N-1} \h^k b_{k}(y,\eta)\right| 
		 \lesssim \h^{N(1-3\delta-12\lambda)+\d}.
	\end{equation}
\end{theorem}

\subsection{Stationary phase asymptotics and the Van Vleck formula}\label{SS:SPandVVcon}
In this section, we prove the bound \eqref{spetaloc}, which as noted, concludes the proof of Theorem \ref{T:VVthm}.  Observe that by \eqref{Phifirstderivs} 
\begin{equation*}
	\Big\{ z=(q,p): p \in \supp(\Theta)  + B(0,\h^\rho),  \frac{\prtl \Phi}{\prtl z}(t,x,y,z) =0, \Im\Phi (t,x,y,z) =0 \Big\} = \tilde{\Xi}_t(x,y),
\end{equation*}
i.e. $\tilde{\Xi}_t(x,y)$ gives the critical set of the phase.

We begin by showing that with $\bar{Y}_t$ as the complex conjugate of the matrix in \eqref{YZdefs}, we have
\begin{equation}\label{PhiHessianDet}
	\det\Big(\frac 1i \frac{\prtl^2\Phi}{\prtl z^2}\Big|_{(q,p) = (y,\eta)} \Big) = \det(i\bar{Y}_t)\det(B_t).
\end{equation}
To see this, we factor the Hessian in \eqref{HessianHK} to see that
\begin{align*}
	\frac 1i \frac{\prtl^2\Phi}{\prtl z^2}\Big|_{(q,p) = (y,\eta)}
	&=
	\begin{bmatrix}
		(A_t+iC_t)^TA_t + I & (A_t+iC_t)^TB_t \\
		(B_t+iD_t)^TA_t-iI & (B_t+iD_t)^TB_t
	\end{bmatrix}
	\\
	&=   \begin{bmatrix}
		(A_t+iC_t)^T  & 0\\
		0 & (B_t+iD_t)^T
	\end{bmatrix}
	\begin{bmatrix}
		A_t+ (A_t+iC_t)^{-T} & I \\
		A_t-i(B_t+iD_t)^{-T} & I
	\end{bmatrix}
	\begin{bmatrix}
		I & 0\\
		0 & B_t
	\end{bmatrix},
\end{align*}
where the matrix functions $A_t(z),B_t(z),\dots$ are all evaluated at $z = (y,\eta)$ (recall from \eqref{siegellemma} that the middle matrix is well-defined).  First consider the middle matrix in the product here, since the matrices in the bottom row commute, its determinant is
computed analogously to a $2\times 2$ matrix, hence it simplifies to $\det(i(B_t+iD_t)^{-T}+(A_t+iC_t)^{-T})$.  The desired identity now follows from
\begin{align*}
	\det\Big(\frac 1i \frac{\prtl^2\Phi}{\prtl z^2}\Big|_{(q,p) = (y,\eta)} \Big) &= \det(A_t+iC_t)\det(B_t+iD_t)\det(i(B_t+iD_t)^{-T}+(A_t+iC_t)^{-T})\det(B_t)\\
	&= \det(i(A_t+iC_t)+(B_t+iD_t))\det(B_t) = \det(i\bar{Y}_t)\det(B_t).
\end{align*}

We now recall \cite[Lemma 3.8]{BilyRobert} showing that $\frac{\prtl^2 S}{\prtl x \prtl y}|_{(t,y,\eta)} = -(B_t(y,\eta))^{-1}$, and reiterate their argument here for convenience of the reader.  Since $\frac{\prtl q_t}{\prtl p}|_{z=(y,\eta)} $ is nonsingular, $x= q_t(y,p)$ determines $p$ as a function of $(x,y)$ near any $\eta \in \tilde{\Xi}_t(x,y)$.  We are now led to observe the two identities 
\begin{equation*}
	\frac{\prtl^2 }{\prtl x \prtl y}\big( S(t,y,p(x,y)) \big) = -	\frac{\prtl  p}{\prtl x} (x,y) = -B_t^{-1}.
\end{equation*}
The second identity just follows from the chain rule, so we show the first of these.  Recall the action $S(t,y,p)$ satisfies $\frac{\prtl S}{\prtl y} = A_t^T p_t - p$, $\frac{\prtl S}{\prtl p} = B_t^T p_t $.  Hence differentiating $S(t,y,p(x,y))$ in $y$ yields
\begin{equation*}
\frac{\prtl }{\prtl y}\big( S(t,y,p(x,y)) \big) = 	
\frac{\prtl S}{\prtl y} + \frac{\prtl S}{\prtl p}\frac{\prtl p}{\prtl x}  = 
\Big(A_t+B_t\frac{\prtl p}{\prtl y} \Big)^Tp_t -p = -p(x,y),
\end{equation*}
where the last identity follows since differentiating both sides of $q_t(y,p(x,y)) = x$ with respect to $y$ gives that $A_t+B_t\frac{\prtl p}{\prtl y} =0$.  Differentiating in $x$ then completes the proof.

At this point, Theorem \ref{T:locasymptotic} follows from an application of Theorem \ref{T:spthm}.  The only matter to address is the branches of the square roots of the factor $\det^{1/2}\bar{Y}_t$ in \eqref{Ldef} and the determinant of the Hessian in \eqref{PhiHessianDet}.  The former is defined by continuity in $t$ so that $\det^{1/2}\bar{Y}_0 = 2^{\frac d2} >0$ when $t=0$.  However, the branch of the square root of \eqref{PhiHessianDet} is determined by the stationary phase theorem, where the branch of $E\mapsto \det^{1/2}E$ is defined on the closure of the space of complex symmetric matrices $E$ with $\Re E$ positive definite, and taken so that $\det^{1/2}E>0$ when $E$ is real (cf. \cite[\S3.4]{HormanderI}).  Nonetheless, we have for some half-integer $\theta_\eta \in \frac 12 \mathbb{Z}$
\begin{equation*}
	\det\!^{1/2} \bar{Y}_t (i^\d \det \bar{Y}_t \det B_t)^{-1/2} = e^{-\frac{i\pi}{2}\theta_\eta }\Big|\det \frac{\prtl^2 S_\eta}{\prtl x \prtl y}\Big|^{1/2}.
\end{equation*}
\bibliographystyle{amsalpha}
\bibliography{bibtexdata}

\providecommand{\MR}[1]{}
\providecommand{\bysame}{\leavevmode\hbox to3em{\hrulefill}\thinspace}
\providecommand{\MR}{\relax\ifhmode\unskip\space\fi MR }
\providecommand{\MRhref}[2]{%
  \href{http://www.ams.org/mathscinet-getitem?mr=#1}{#2}
}
\providecommand{\href}[2]{#2}
\begin{thebibliography}{ABZ17}

\bibitem[ABZ17]{AlazardBurqZuilyStatPhase}
T.~Alazard, N.~Burq, and C.~Zuily, \emph{A stationary phase type estimate},
  Proc. Amer. Math. Soc. \textbf{145} (2017), no.~7, 2871--2880. \MR{3637937}

\bibitem[B{\'e}r77]{Berard77}
Pierre~H. B{\'e}rard, \emph{On the wave equation on a compact {R}iemannian
  manifold without conjugate points}, Math. Z. \textbf{155} (1977), no.~3,
  249--276. \MR{0455055}

\bibitem[Bon17]{BonthonneauTheta}
Yannick Bonthonneau, \emph{The {$\Theta$} function and the {W}eyl law on
  manifolds without conjugate points}, Doc. Math. \textbf{22} (2017),
  1275--1283. \MR{3690266}

\bibitem[BR01]{BilyRobert}
J.~M. Bily and D.~Robert, \emph{{The semi-classical {V}an {V}leck formula.
  {A}pplication to the {A}haronov-{B}ohm effect}}, {Long time behaviour of
  classical and quantum systems ({B}ologna, 1999)}, {Ser. Concr. Appl. Math.},
  vol.~1, World Sci. Publ., River Edge, NJ, 2001, pp.~89--106. \MR{1852218}

\bibitem[BR02]{BouzouinaRobert}
A.~Bouzouina and D.~Robert, \emph{Uniform semiclassical estimates for the
  propagation of quantum observables}, Duke Math. J. \textbf{111} (2002),
  no.~2, 223--252. \MR{1882134}

\bibitem[CG20]{canzani2020weyl}
Yaiza Canzani and Jeffrey Galkowski, \emph{Weyl remainders: an application of
  geodesic beams}, arXiv preprint arXiv:2010.03969 (2020).

\bibitem[Cha74]{ChazarainPoisson}
J.~Chazarain, \emph{Formule de {P}oisson pour les vari\'{e}t\'{e}s
  riemanniennes}, Invent. Math. \textbf{24} (1974), 65--82. \MR{343320}

\bibitem[CR97]{CombescureRobertWavePackets}
M.~Combescure and D.~Robert, \emph{{Semiclassical spreading of quantum wave
  packets and applications near unstable fixed points of the classical flow}},
  Asymptot. Anal. \textbf{14} (1997), no.~4, 377--404. \MR{1461126 (98g:81040)}

\bibitem[CR12]{CombescureRobertBook}
Monique Combescure and Didier Robert, \emph{{Coherent states and applications
  in mathematical physics}}, {Theoretical and Mathematical Physics}, Springer,
  Dordrecht, 2012. \MR{2952171}

\bibitem[DG75]{DuistermaatGuillemin}
J.~J. Duistermaat and V.~W. Guillemin, \emph{{The spectrum of positive elliptic
  operators and periodic bicharacteristics}}, Invent. Math. \textbf{29} (1975),
  no.~1, 39--79. \MR{0405514 (53 \#9307)}

\bibitem[Gre54]{GreenConjugatePoints}
L.~W. Green, \emph{{Surfaces without conjugate points}}, Trans. Amer. Math.
  Soc. \textbf{76} (1954), 529--546. \MR{0063097}

\bibitem[Gut71]{gutzwiller1971periodic}
Martin~C Gutzwiller, \emph{Periodic orbits and classical quantization
  conditions}, Journal of Mathematical Physics \textbf{12} (1971), no.~3,
  343--358.

\bibitem[HJ00]{HagedornJoye}
George~A. Hagedorn and Alain Joye, \emph{{Exponentially accurate semiclassical
  dynamics: propagation, localization, {E}hrenfest times, scattering, and more
  general states}}, Ann. Henri Poincar{\'e} \textbf{1} (2000), no.~5, 837--883.
  \MR{1806980 (2001k:81066)}

\bibitem[H{\"o}r90]{HormanderI}
Lars H{\"o}rmander, \emph{{The analysis of linear partial differential
  operators. {I}}}, second ed., {Grundlehren der Mathematischen Wissenschaften
  [Fundamental Principles of Mathematical Sciences]}, vol. 256,
  Springer-Verlag, Berlin, 1990, Distribution theory and Fourier analysis.
  \MR{1065993}

\bibitem[LS00]{LaptevSigal}
A.~Laptev and I.~M. Sigal, \emph{Global {F}ourier integral operators and
  semiclassical asymptotics}, Rev. Math. Phys. \textbf{12} (2000), no.~5,
  749--766. \MR{1767504}

\bibitem[Mei92]{MeinrenkenGutzwillerTrace}
Eckhard Meinrenken, \emph{Semiclassical principal symbols and {G}utzwiller's
  trace formula}, Rep. Math. Phys. \textbf{31} (1992), no.~3, 279--295.
  \MR{1232640}

\bibitem[MF81]{MaslovFedoriuk}
V.~P. Maslov and M.~V. Fedoriuk, \emph{Semiclassical approximation in quantum
  mechanics}, Mathematical Physics and Applied Mathematics, vol.~7, D. Reidel
  Publishing Co., Dordrecht-Boston, Mass., 1981, Translated from the Russian by
  J. Niederle and J. Tolar, Contemporary Mathematics, 5. \MR{634377}

\bibitem[OL20]{OhLeeUniform}
Sewook Oh and Sanghyuk Lee, \emph{Uniform stationary phase estimate with
  limited smoothness}, arXiv preprint arXiv:2012.12572 (2020).

\bibitem[Rob87]{RobertSemiClassique}
Didier Robert, \emph{Autour de l'approximation semi-classique}, Progress in
  Mathematics, vol.~68, Birkh\"{a}user Boston, Inc., Boston, MA, 1987.
  \MR{897108}

\bibitem[Rob10]{RobertHK}
\bysame, \emph{{On the {H}erman-{K}luk semiclassical approximation}}, Rev.
  Math. Phys. \textbf{22} (2010), no.~10, 1123--1145. \MR{2740707}

\bibitem[RS93]{RobbinSalamonPaths}
Joel Robbin and Dietmar Salamon, \emph{The {M}aslov index for paths}, Topology
  \textbf{32} (1993), no.~4, 827--844. \MR{1241874}

\bibitem[Rud76]{RudinPrinciples}
Walter Rudin, \emph{{Principles of mathematical analysis}}, third ed.,
  McGraw-Hill Book Co., New York-Auckland-D{\"u}sseldorf, 1976, International
  Series in Pure and Applied Mathematics. \MR{0385023}

\bibitem[SR09]{SwartRousse}
Torben Swart and Vidian Rousse, \emph{A mathematical justification for the
  {H}erman-{K}luk propagator}, Comm. Math. Phys. \textbf{286} (2009), no.~2,
  725--750. \MR{2472042}

\bibitem[SVT12]{SchubertValTos}
Roman Schubert, Ra{\'u}l~O. Vallejos, and Fabricio Toscano, \emph{{How do wave
  packets spread? {T}ime evolution on {E}hrenfest time scales}}, J. Phys. A
  \textbf{45} (2012), no.~21, 215307, 28. \MR{2925343}

\bibitem[Tac20]{TacyStatPhase}
M.~Tacy, \emph{Stationary phase type estimates for low symbol regularity},
  Anal. Math. \textbf{46} (2020), no.~3, 605--617. \MR{4137136}

\bibitem[Vol90]{Volovoy}
A.~V. Volovoy, \emph{Improved two-term asymptotics for the eigenvalue
  distribution function of an elliptic operator on a compact manifold}, Comm.
  Partial Differential Equations \textbf{15} (1990), no.~11, 1509--1563.
  \MR{1079602}

\bibitem[VV28]{VanVleckCorrespondence}
John~H. Van~Vleck, \emph{The correspondence principle in the statistical
  interpretation of quantum mechanics}, Proceedings of the National Academy of
  Sciences of the United States of America \textbf{14} (1928), no.~2, 178.

\bibitem[Zwo12]{ZworskiSemiclassicalAnalysis}
Maciej Zworski, \emph{{Semiclassical analysis}}, {Graduate Studies in
  Mathematics}, vol. 138, American Mathematical Society, Providence, RI, 2012.
  \MR{2952218}

\end{thebibliography}
\end{document}